\documentclass[reqno]{amsart}

\usepackage{eucal}
\usepackage{latexsym}
\usepackage{amsmath}
\usepackage{amssymb}
\usepackage{amscd}
\usepackage{multicol}
\usepackage{amsfonts}
\usepackage{amsthm}
\usepackage{xcolor}
\usepackage{color}
\usepackage{graphicx}
\usepackage[all]{xy}
\usepackage{epsfig}
\usepackage{psfrag}
\usepackage{hyperref}
                                                                                
\input xy
\xyoption{all}

\theoremstyle{plain}
\newtheorem{thm}{Theorem}
\newtheorem{theorem}[thm]{Theorem}
\newtheorem{corollary}[thm]{Corollary}
\newtheorem{lemma}[thm]{Lemma}
\newtheorem{prop}[thm]{Proposition}
\newtheorem{defin}[thm]{Definition}

\newtheoremstyle{exm}{9pt}{9pt}{}{}{\bfseries}{}{.5em}{}
\theoremstyle{exm}
\newtheorem{exm}[thm]{Example}
\newtheoremstyle{rmk}{9pt}{9pt}{}{}{\bfseries}{}{.5em}{}
\theoremstyle{rmk}
\newtheorem{rmk}[thm]{Remark}
\theoremstyle{alg}

\newtheoremstyle{question}{9pt}{9pt}{}{}{\bfseries}{}{.5em}{}
\theoremstyle{question}

\numberwithin{equation}{section}
\numberwithin{thm}{section}
\numberwithin{figure}{section}

\newcommand{\g}{\mathfrak{t}}
\newcommand{\gd}{\mathfrak{t}^{*}}

\newcommand{\il}{\mathsf{i}}
\newcommand{\ee}{\mathsf{e}_{\T}^-}
\newcommand{\ii}{\textrm{Ind}}

\newcommand{\pp}{\mathbb{P}}

\newcommand{\e}[1]{e^{2 \pi \il {#1}}}
\newcommand{\p}{\textrm{pt}}
\newcommand{\KK}{\mathbb{K}}
\newcommand{\Ke}{\mathcal{K}_\T}
\newcommand{\K}{\mathcal{K}_{\T}(M)}
\newcommand{\KT}{\mathcal{K}_{\T}(M^{\T})}
\newcommand{\T}{\mathbb{T}}
\newcommand{\Mt}{M^{\T}}
\renewcommand{\SS}{{\mathbb{S}}}
\newcommand{\R}{\mathbb{R}}
\newcommand{\Z}{{\mathbb{Z}}}
\newcommand{\C}{{\mathbb{C}}}

\newcommand{\Q}{\mathbb{Q}}

\title[Canonical bases]{Canonical bases for the equivariant cohomology and K-theory rings of symplectic toric manifolds}

\author[M. Pabiniak]{Milena Pabiniak}
\address{Departamento de Matem\'{a}tica, Centro de An\'{a}lise Matem\'{a}tica, Geometria e Sistemas Din\^{a}micos-LARSYS, Instituto Superior T\'ecnico, Av. Rovisco Pais 1049-001 Lisbon, Portugal}
\email{milenapabiniak@gmail.com}

\author[S. Sabatini]{Silvia Sabatini}
\address{Mathematisches Institut, Universit\"at zu K\"oln, Weyertal 86-90, D-50931 K\"oln, Germany}
\email{sabatini@math.uni-koeln.de}

\date{\today}

\begin{document}
\begin{abstract}
Let $M$ be a symplectic toric manifold acted on by a torus $\mathbb{T}$.
In this work we exhibit an explicit basis for the
equivariant K-theory ring $\mathcal{K}_{\mathbb{T}}(M)$ which is canonically associated
to a generic component of the moment map. We provide
a combinatorial algorithm for computing the restrictions of the elements of this basis
to the fixed point set; these, in turn, determine the ring structure of $\mathcal{K}_{\mathbb{T}}(M)$.
The construction is based on the notion of local index 
 at a fixed point, similar to that introduced by Guillemin and Kogan in \cite{GK}.

We apply the same techniques to exhibit an explicit basis for the equivariant cohomology ring $H_\T(M; \mathbb{Z})$ which is canonically
associated to a generic component of the moment map. Moreover we prove that the elements of this basis coincide
with some well-known sets of classes: the equivariant Poincar\'e duals to the closures of unstable manifolds, and also
 the canonical classes introduced by
Goldin and Tolman in \cite{GT}, which exist whenever the moment map is index increasing.
\end{abstract}
\maketitle

\tableofcontents
\section{Introduction}\label{intro} 

Let $(M^{2n},\omega)$ be a compact symplectic toric manifold of dimension $2n$, i.e.\ a compact symplectic manifold
equipped with an effective Hamiltonian action of an $n$-dimensional torus $\T$ with Lie algebra $\mathfrak{t}$, and let $\psi\colon M\to \mathfrak{t}^*$
be a moment map for the action. Moreover assume that the fixed point set $M^\T$ is discrete.
 Let $\mathbb{H}_\T(M)$ denote  
either the equivariant cohomology ring with $\Z$ coefficients or the equivariant K-theory ring of $M$. 
The inclusion $i\colon \Mt \hookrightarrow M$ is $\T$-equivariant, hence it induces a map 
$$i^*\colon \mathbb{H}_\T(M) \to \mathbb{H}_\T(M^\T)\,.$$
Since $M^\T$ is discrete the map $i^*$ is always injective (cf.\ \cite{Ki} for a proof in the equivariant cohomology setting and \cite{GK} for the equivariant K-theory ring). Therefore $\mathbb{H}_\T (M)$ can be viewed as a subring of $\mathbb{H}_\T(M^\T)$.

The main goal of this paper is to analyze the ring structure of $\mathbb{H}_\T (M)$ and construct
 an explicit basis of $\mathbb{H}_\T(M)$ canonically associated to a generic component $\mu$ of the moment map $\psi$ (see page \pageref{ordering fixed points}). The restrictions of the basis elements to the fixed points, i.e. their images in $\mathbb{H}_\T(M^\T)$, determine the equivariant structure constants. 
 
 A result of Kirwan guarantees that there always exists a basis for $\mathbb{H}_\T(M)$ associated to a generic component of the moment map, the elements of this basis being called \emph{Kirwan classes} (see \cite{Ki} for the equivariant cohomology setting, and Proposition \ref{existence kirwan} for a generalization of this idea to equivariant K-theory). This basis is, however, not unique. 
 Several authors have added different conditions that would ensure this basis to be unique, i.e. to be \emph{canonically} associated to a generic component of the moment map. 
 For example, Guillemin and Zara \cite{GZ,GZgeneratingfamilies} study this problem for the equivariant cohomology ring of \emph{GKM spaces} (see Section \ref{GKM section}).
 The elements of the basis they introduce
are called equivariant Thom classes, and should be thought of as the ``equivariant Poincar\'e duals" to the closures
of the unstable manifolds $W^u(p)$ of a generic component $\mu$ of the moment map with respect to an invariant metric.
When such closures are smooth, these equivariant Poincar\'e duals can be computed explicitly in terms of
the $\T$ representations on the normal bundle of $\overline{W^u(p)}$. However in general such manifolds are not smooth, and in this case
the restrictions of the equivariant Thom classes to the fixed point set are analyzed by means of the combinatorics of the GKM graph.
 In \cite{GT} Goldin and Tolman study a similar problem on 
Hamiltonian $\T$-spaces, and introduce basis elements for the equivariant cohomology ring which are
canonically associated to a generic component $\mu$ of the moment map.  However in both cases, in order for the basis elements to exist and be unique,
there is an essential requirement on $\mu$: it needs to be \emph{index increasing} (see Definition \ref{ii and nii}). This is satisfied for example when
the stable and unstable manifolds of $\mu$ with respect to an invariant metric meet transversally. 

In \cite{GK} Guillemin and Kogan introduce equivariant K-theory classes which are a basis of the equivariant K-theory ring $\K$
of a Hamiltonian $\T$-space (as a module over the equivariant K-theory of a point). 
However no explicit connection is given between such basis and the ``natural" basis given by the K-theoretical equivariant Poincar\'e duals
to the closures of the unstable manifolds (in the case in which there exists an invariant metric for which these are all smooth).  
To express the extra conditions imposed on Kirwan classes, Guillemin and Kogan introduce the \emph{local index map}, that associates to each class in 
$\K$ and each fixed point $q \in M^\T$ an element of $\mathcal{K}_\T(\p)$. Note that $\mathcal{K}_\T(\p)$ can be identified with $R(\T)$,
the representation ring of $\T$.

Inspired by this idea, we present a slightly different definition of local index of an equivariant K-theory class $\tau$ at a fixed point $q$ $$\ii_q \colon \K \to R(\T),$$
and give an explicit combinatorial recipe for computing it 
 (see Definition \ref{definition local index} in Section \ref{section local index}).
Using this notion we introduce a basis for the equivariant K-theory ring of a symplectic toric manifold $M$ which is canonically associated to a generic component of the moment map, both in the index increasing and non-index increasing case. 
The main result of the paper is the following. Let $F_p$ be the \emph{flow-up} manifold at $p\in M^\T$ (see page~\pageref{def fq}, Section \ref{section local index}),
corresponding to the closure of the unstable manifold at $p$.
\begin{theorem}
Let $(M,\omega,\T,\psi)$ be a symplectic toric manifold of dimension $2n$, together with 
a choice of a generic component of the moment map $\mu \colon M\to \R$. Let $\K$ be the equivariant K-theory ring of $M$.
Then for each $p\in \Mt$ there exists a unique Kirwan class $\tau_p \in \K$, called the {\bf i-canonical class} at the fixed point $p$
(see Definition \ref{definition canonical classes}),
 satisfying
\begin{enumerate}
\item $\ii_q(\tau_p)=1$ for all points $q \in F_p\cap M^\T$;
\item $\ii_q(\tau_p)=0$ for all points $q \notin F_p\cap M^\T$.
\end{enumerate}
Moreover, the set $\{\tau_p\}_{p\in M^\T}$ is a basis for $\Ke(M)$ as a module over $R(\T)$.
\end{theorem}
Being a Kirwan class means that $\tau_p(p)=\ee(p)$, the equivariant (K-theoretical) Euler class of the negative normal bundle $N_p^-$ of $\mu$ at $p$, and that $\tau_p(q)=0$ for all $q \in M^\T$ with $\mu(q)<\mu(p)$ (see Proposition \ref{existence kirwan}).

We prove the claims of the above theorem in two separate Propositions: 
uniqueness is proved in Proposition \ref{uniqueness},
existence in Proposition \ref{existence canonical}.

Moreover, we show that in the index increasing case these classes are indeed the equivariant Poincar\'e duals to the flow-up manifolds $F_p$ (see Lemma \ref{def kirwanclasses} and Prop.\ \ref{classes from cohomology are canonical}).

Note that we require our classes $\tau_p$ to have local index $1$ on $F_p$, not only at $p$ as in \cite{GK}. 
As a consequence, the trivial bundle $\mathbf{1}\in \K$ is an element of our basis of i-canonical classes, while it does not belong to the basis exhibited in \cite{GK}.
Moreover, when $M$ is a complex projective space endowed with the standard toric action, the basis of i-canonical classes consists of powers of 
the (equivariant) prequantization line bundle (see Example \ref{projective space}).
Another important advantage of our approach is that 
the local index, and thus also the 
i-canonical classes, is easy to calculate directly from the combinatorics underlying the symplectic toric manifold, as we demonstrate by various examples. 
We give explicit formulas for the elements of this basis when the component of the moment map is index increasing,
 and inductive formulas otherwise. 

The definition of local index can also be translated to the equivariant cohomology setting (Section \ref{ec}), thus allowing to define a
 canonical basis for the equivariant cohomology ring of $M$.
\begin{theorem}\label{main cohomology}
Let $(M,\omega,\T,\psi)$ be a symplectic toric manifold of dimension $2n$, together with 
a choice of a generic component of the moment map $\mu \colon M\to \R$. 
Let $H_\T^*(M;\Z)$ be the equivariant cohomology ring of $M$ with $\Z$-coefficients.
Then for each $p\in \Mt$ there exists a unique Kirwan class $\tau_p \in H_\T^*(M;\Z)$, called the {\bf i-canonical class} at the fixed point $p$
(see Definition \ref{definition canonical classes cohom}), satisfying
\begin{enumerate}
\item $\ii_p(\tau_p)=1$;
\item $\ii_q(\tau_p)=0$ for all points $q \in \Mt \setminus \{p\}$.
\end{enumerate}
The set $\{\tau_p\}_{p\in M^\T}$ is a basis for $H_\T^*(M;\Z)$ as a module over $H_\T^*(\p;\Z)$.
\end{theorem}
Indeed, in Proposition \ref{prop classes equal} we prove that the i-canonical classes in equivariant cohomology coincide with the
equivariant Poincar\'e duals to the flow-up manifolds $F_p$. Moreover, when the chosen component of the moment map is index increasing,
they also coincide with the canonical classes introduced by Goldin and Tolman in \cite{GT}
(see Proposition \ref{pd is gt}).

Note that the definition of i-canonical classes in equivariant cohomology is not a direct translation of the definition in K-theory.
Here we want the local index of $\tau_p$ to vanish at all fixed points other than $p$. The reason for this difference is that we want the class $\mathbf{1} \in H_\T^*(M;\Z)$ to be one of the elements of the basis of $H_\T^*(M;\Z)$.
\\ 

{\bf Organization.}
Section \ref{background} contains the background material and some preliminary results. In Section \ref{section construction} we construct i-canonical classes 
in equivariant K-theory, thus proving their existence.
Section \ref{ec} proves similar results in the
equivariant cohomology setting. We finish the paper with an appendix about an explicit description of the Kirwan map.

{\bf Acknowledgements.} The first author was supported by the Funda\c{c}\~ao para a Ci\^encia e a Tecnologia, Portugal: postdoctoral fellowship SFRH/BPD/87791/2012 and
projects EXCL/MAT-GEO/0222/2012, PTDC/MAT/117762/2010. 
The second author was supported by the Funda\c{c}\~ao para a Ci\^encia e a Tecnologia, Portugal, from 09/2013 until 08/2014: postdoctoral fellowship SFRH/BPD/86851/2012 and projects EXCL/MAT-GEO/0222/2012, PTDC/MATH/117762/2010.

\section{Background and preliminary results}\label{background}
\subsection{Hamiltonian spaces}\label{sec:1}
Let $(M,\omega)$ be a compact symplectic manifold of dimension $2n$, and $\T$ a compact real torus of dimension $d$ with Lie algebra
$\g$. Suppose that $\T$
acts on $(M,\omega)$ in a Hamiltonian fashion with moment map $\psi$, i.e.\ there exists a $\T$-invariant map $\psi\colon M\to \gd$ 
satisfying
\begin{equation}\label{moment map}
\iota_{\xi^\#}\omega=-d\langle \psi,\xi \rangle\;,
\end{equation}
where $\xi^\#$ denotes the vector field on $M$ associated with the flow of symplectomorphisms generated by $\xi\in \g$, and $\langle \cdot , \cdot \rangle$ the dual pairing between $\gd$ and $\g$. Unless otherwise stated, we assume the action to be effective, and the fixed point set $M^\T$ of the action to be discrete. We
refer to $(M,\omega,\T,\psi)$ as a {\bf Hamiltonian $\T$-space}.
Recall that when $\dim(\T)=\frac{\dim(M)}{2}$, the Hamiltonian
$\T$-space $(M,\omega,\T,\psi)$ is called a {\bf symplectic toric manifold}. 
Before specializing to the case of symplectic toric manifolds, we first introduce some notions that will be used throughout this note and do not
depend on the action being toric.

Let $\K$ denote the equivariant K-theory ring of $M$, i.e.\ the abelian group associated to the semigroup of 
isomorphism classes of complex $\T$-vector bundles over $M$, endowed with the direct sum operation 
$\oplus$ and the tensor product $\otimes$.
Thus, if $M$ is a point, $\mathcal{K}_\T(\p)$ is the representation ring of the torus $\T$, henceforth denoted by
$R(\T)$. Observe that, if $\ell^*\subset \gd$ denotes the weight lattice of $\gd$
and $\ell^*=\Z\langle x_1,\ldots,x_d\rangle$, then
$R(\T)$ can be identified with the ring of finite sums 
$\big\{\sum_{j\in J}n_je^{2\pi \il w_j}\;\;\mbox{s.t.}\;\;  |J|<\infty, n_j\in \Z \;\;\mbox{and}\;\; w_j\in \ell^*\big\} $, or equivalently 
\begin{equation}\label{id R(T)}
R(\T)=\Z[\e{x_1},\ldots,\e{x_d},\e{(-x_1-x_2-\ldots -x_d)}]
\end{equation}
 i.e. $R(\T)$ is identified with the character ring of $\T$. The unique map $\pi\colon M\to \{\p\}$ induces a map $\pi^*\colon R(\T)\to \K$ which gives
 $\K$ the structure of an $R(\T)$-module.
 
Observe that the inclusion $i\colon \Mt \hookrightarrow M$, which is clearly $\T$-equivariant, gives rise to a map in equivariant K-theory:
$$i^*\colon \K \to \KT\;.$$
Since we assumed $\Mt$ to be discrete we have $\KT=\bigoplus_{p\in \Mt}R(\T)$, which can be regarded as the ring of maps
that assign to each fixed point $p\in \Mt$ a representation in $R(\T)$.
In \cite[Corollary 2.2]{GK}, the authors prove that for (compact) Hamiltonian $\T$-spaces with discrete fixed point set $M^\T$, the above restriction map $i^*$ 
is {\bf injective} (this result is quoted here in Theorem \ref{facts K theory}; see also \cite[Theorem 2.5]{HL2} for the case in which $M$ is not necessarily compact). 
Thus $\K$ can be regarded as a subring of  $\mathcal{K}_\T(M^\T)$, which is a much easier object to deal with. 
 Henceforth we will always identify $\Ke(M)$ with $i^*(\Ke(M))$:
 $$\K \cong i^*(\Ke(M))\subset \Ke(M^\T)\cong \bigoplus_{p\in \Mt}R(\T).$$ 
Let  $p\in M^\T$ and $i_p^*\colon \K\to \mathcal{K}_\T(\{p\})$ the map induced by the inclusion $i_p\colon \{p\} \hookrightarrow M$. 
For every $\tau\in \K$ we denote by $\tau(p)\in R(\T)$ the value $i_p^*(\tau)$, and define the {\bf support} of $\tau$ to be
$$\textrm{supp}(\tau)=\{q \in M^\T\,|\,\tau(q) \neq 0\}\subset M^\T.$$
Observe that injectivity of $i^*$ implies that $\tau=0 \in \K$ if and only if $\tau(p)=0$ for all $p \in M^\T$, or equivalently if and only if $\textrm{supp}(\tau)=\emptyset$.
 
Let $J\colon TM\to TM$ be a $\T$-invariant almost complex structure compatible with $\omega$, i.e. $\omega(J\cdot,\cdot)$ defines an (invariant) inner product on $M$. If $p$ is a fixed point of the action, the $\T$-action on $M$ induces a representation on $T_pM\simeq \C^n$, called the {\bf isotropy representation of $\T$ at $p$ }, which is given by
\begin{equation}
\exp(\xi)\cdot (z_1,\ldots,z_n)= (e^{2\pi\il w_1(\xi)}z_1,\ldots,e^{2\pi\il w_n(\xi)}z_n),\quad\mbox{for every}\quad \xi\in \g\;.
\end{equation}
Here the $w_i$'s are well-defined nonzero elements of $\ell^*$ and are called the {\bf weights (of the isotropy representation of $\T$) at $p$}. The weights are nonzero because we are assuming $M^\T$ to be discrete, and the isotropy action commutes with the $\T$-action on $M$ around $p$. The set of these weights, counted with multiplicities, will be denoted by $W_p$, and the set of all isotropy weights, counted with multiplicities, by $W=\coprod_{p\in M^\T}W_p$.

Take a $\overline{\xi}\in \g$ generating a circle subgroup $S^1=\{\exp(t\,\overline{\xi})\,;\,t \in \R\}\subset \T$, such that $w(\overline{\xi})\neq 0$ for all $w\in W$. 
It is well-known (cf.\ \cite{F})
that the $\overline{\xi}$-component of the moment map $\mu= \psi^{\overline{\xi}}\colon M\to \R$, defined as $\mu(q)=\langle \psi(q),\overline{\xi}\rangle$, is a $\T$-invariant Morse function whose critical set coincides with the fixed point set $M^\T$. It is easy to check that equation \eqref{moment map} implies that the isotropy weights in the representation on the negative (resp.\ positive) normal bundle of $\mu$ at $p$, denoted by $N_p^-$ (resp.\ $N_p^+$), coincide with the {\bf positive (resp.\ negative) weights}, i.e.\ with those $w$'s in $W_p$ such that $w(\overline{\xi})>0$ (resp. $w(\overline{\xi})<0$). We denote this (multi)set by $W_p^+$ (resp. $W_p^-$).
Observe that, with this convention, all the weights at the minimum $p_0$ of $\mu$ are negative, so that $W_{p_0}=W_{p_0}^-$.

Let $\lambda_p$ denote the number of positive weights at $p$ for every $p\in M^\T$; then the Morse index of $\mu$ at $p$
is precisely $2\lambda_p$.
\begin{defin}\label{eec}
Given $p\in M^\T$, the {\bf equivariant (K-theoretical) Euler class} of the negative normal bundle $N_p^-$ of $\mu$ at $p$, denoted by
$\ee(p)$, is an element of $\mathcal{K}_\T(\{p\})$ defined as
$$
\ee(p)=\prod_{w_j\in W_p^+}(1-e^{2\pi \il w_j}) 
$$
\end{defin}
We will see that this class plays a key role in the construction of i-canonical classes.

We say that $\mu$ {\bf separates fixed points} if $\mu(p)\neq \mu(q)$ for every $p,q\in M^\T$ with $p\neq q$.
Observe that, since we only deal with symplectic toric manifolds, the moment map for the $\T$ action is always injective when restricted to the fixed point set. Thus, for an open dense subset of $\overline{\xi} \in \g$, the corresponding $\mu$ separates fixed points.
We call $\overline{\xi} \in \g$ {\bf generic} if $w(\overline{\xi})\neq 0$ for all $w\in W$ and the corresponding $\mu$ separates fixed points. Such $\mu$ is called a {\bf generic
component of the moment map}.
Henceforth, we order the fixed points of the action $p_0,\ldots,p_N$ in such a way that $\mu(p_0)<\mu(p_1)<\cdots < \mu(p_N)$ and denote this ordering by 
\begin{equation}\label{ordering fixed points}
 p_0\prec p_1\prec \cdots \prec p_N\,.
\end{equation}
The following proposition is not new, but since it plays a key role in our work, we include the proof for the readers' convenience.
\begin{prop}\label{existence kirwan}
Let $(M,\omega,\T,\psi)$ be a Hamiltonian $\T$-space, and let $\mu\colon M\to\R$ be a generic component of the moment map.
Then for every $p\in M^\T$ there exists a class $\nu_p\in \K$, called a {\bf Kirwan class} at $p$,  such that 
\begin{itemize}
\item[(i)] $\nu_p(p)=\ee(p)$;
\item[(ii)] $\nu_p(q)=0$ for every $q\in M^\T$ such that $q \prec p$, i.e. $\mu(q)< \mu(p)$.
\end{itemize}
Moreover the set $\{\nu_p\}_{p\in M^\T}$ is a basis for $\K$ as an $R(\T)$-module.
\end{prop}

Before giving the proof of this Proposition, we recall here a few important facts about the equivariant K-theory ring of Hamiltonian $\T$-spaces.
For every $p\in M^\T$ and $\varepsilon>0$ sufficiently small, let 
$$
M_p^\pm = \{q\in M\mid \mu(q)\leq \mu(p)\pm \varepsilon\}.
$$
By a standard Morse-theoretic argument there exists an equivariant homotopy equivalence between $(M_p^+,M_p^-)$ and $(D^{2\lambda_p},\partial D^{2\lambda_p})$,
where $D^{2\lambda_p}$ is the disk of dimension $2\lambda_p$ centered at $p$; hence $\Ke(M_p^+,M_p^-)\simeq \Ke(D^{2\lambda_p},\partial D^{2\lambda_p})$. 

Consider the following diagram:
\begin{equation}\label{diag thom}
 \xymatrix{ 
\Ke(M_p^+,M_p^-) \ar[r]^-{\alpha_p} &  \Ke(M_p^+)\ar[d]_{\iota_p^*}\\
\Ke(\{p\})\ar[u]_{\mathcal{T}_p}  \ar[r] &  \Ke(\{p\}) \\
}
\end{equation}
where $\mathcal{T}_p$ is the Thom isomorphism, $\alpha_p$ the map in the long exact sequence of the pair $(M_p^+,M_p^-)$ and $\iota_p^*$ the restriction map. We have that $\iota_p^*\circ \alpha_p \circ \mathcal{T}_p$ is just the multiplication by $\ee(p)$, which is not a zero divisor in $\Ke(\{p\})$, thus implying that $\alpha_p$ is injective 
for every $p\in M^\T$. This is the main ingredient of the following Theorem (whose proof is omitted here, but the reader can refer to \cite[Lemma 2.1 and Corollary 2.2]{GK}):
\begin{thm}\label{facts K theory}
For every $p\in M^\T$, the K-theory long exact sequence of the pair $(M_p^+,M_p^-)$ splits into short exact sequences
\begin{equation}\label{ses}
\xymatrix{
0\ar[r] &\Ke(M_p^+,M_p^-) \ar[r]^{\alpha_p} & \Ke(M_p^+)\ar[r]^{\beta_p}& \Ke(M_p^-)\ar[r]& 0
}
\end{equation}
Moreover, the following map 
\begin{equation}\label{inj}
 \Ke(M_p^\pm)\to \Ke(M_p^\pm \cap M^\T)
\end{equation}
is injective for every $p\in M^\T$, hence so is
\begin{equation}\label{inj 2}
i^*\colon \Ke(M)\to \Ke(M^\T),
\end{equation}
and 
\begin{equation}\label{surj}
\Ke(M)\to \Ke(M_p^\pm)
\end{equation}
is surjective for every $p\in M^\T$.
\end{thm}
\begin{proof}[Proof of Proposition \ref{existence kirwan}]

Consider any K-theory class $\nu$ in $\Ke(M_p^+)$. Recall that we denote $\iota_p^*(\nu)$ by $\nu(p)$.
By the exactness of \eqref{ses} and the analysis of the diagram \eqref{diag thom} done before, we obtain that $\nu$ is in
$\ker(\beta_p)$ if and only if it satisfies $\nu(p)=f\,\ee(p)$ 
 for some $f\in R(\T)=\Ke(\{p\})$ and $\nu(q)=0$ for all $q\in M_p^-\cap M^\T$. 
By Theorem \ref{facts K theory} the restriction map $\Ke(M_p^+)\to \Ke(M_p^+\cap M^\T)$ is injective, so once $f$ is fixed
 these conditions uniquely characterize the class $\nu$. By taking $f=1$ and extending the class $\nu$ to $M$, which can be achieved by using the surjectivity of \eqref{surj}, we obtain a class $\nu_p\in \Ke(M)$ satisfying properties (i) and (ii) in Proposition \ref{existence kirwan}, henceforth called a Kirwan class.  
 
 Consider a collection of Kirwan classes $\{\nu_p\}_{p\in M^\T}$. We first need to prove that they generate
 $\Ke(M)$ as an $R(\T)$-module.  
 Let $\gamma\in \Ke(M)$, and let $q_0$ be the first fixed point (in $\prec$ order) where $\gamma(q_0)\neq 0$. Since the restriction of $\gamma$ to
 $M_{q_0}^-$ is zero, from what we observed before,  and by property (i) of $\nu_{q_0}$, we have $\gamma(q_0)=f_0 \,\ee(q_0)=f_0 \,\nu_{q_0}(q_0)$, for some $f_0\in R(\T)$.
 Thus the class $\gamma - f_0 \nu_{q_0}$ is zero at $q_0$, and by property (ii) of $\nu_{q_0}$, the first fixed point $q_1$ where it doesn't vanish satisfies $q_0\prec q_1$.
 By repeating this argument we can construct a class $\gamma-\sum_{i=1}^m f_i \nu_{q_i}$, with $f_i\in R(\T)$ for all $i=1,\ldots,m$, whose restriction to the fixed
 point set vanishes identically. By the injectivity of \eqref{inj 2} it follows that $\gamma-\sum_{i=1}^m f_i \nu_{q_i}=0$, and hence $\{\nu_p\}_{p\in M^\T}$ is a set of generators of $\Ke(M)$
 as an $R(\T)$-module. 
 
 Now suppose that $\delta = \sum_{j=0}^s c_j \nu_{p_{i_j}}=0$, where
 $c_j\in R(\T)$ and $c_j\neq 0 $ for every $j=0,\ldots,s$, and assume that $i_1<i_2<\cdots < i_s$. 
 Observe that $\nu_{p_{i_1}}$ is the only class that does not vanish at $p_{i_1}$, and hence $\delta(p_{i_1})= c_{1}\nu_{p_{i_1}}(p_{i_1})=c_1 \ee(p_{i_1})\neq 0$, which gives a contradiction. 
 We conclude that the set $\{\nu_p\}_{p\in M^\T}$ is a basis for $\Ke(M)$
 as an $R(\T)$-module.
 \end{proof}
 Finally we recall the following Lemma, whose proof follows, \emph{mutatis mutandis}, from that of \cite[Lemma 2.4]{GoSa}.
\begin{lemma}\label{structure constants}
With the same hypotheses of Proposition \ref{existence kirwan}, let $\{\nu_p\}_{p\in M^\T}$ be
 a collection of Kirwan classes. Then there exists an explicit algorithm that computes the equivariant structure constants 
 associated to the basis $\{\nu_p\}_{p\in M^\T}$ from their 
restrictions to the fixed point set $\{i^*(\nu_p)\}_{p\in M^\T}$.
\end{lemma}  
Observe that Kirwan classes are never unique, unless $M$ is a point. Indeed, if
 $\nu_p$ is a Kirwan class at $p$, 
then the class $$\nu_p+\sum_{\{q\in M^\T \mid \mu(q)>\mu(p)\}}a_q\nu_q$$
also is, for any set of $a_q\in R(\T)$.

In the next sections we introduce ``special" Kirwan classes, i.e.\ Kirwan classes satisfying some extra assumptions that will ensure their uniqueness, and compute their restrictions to the fixed point set. 

 \subsection{The equivariant K-theory ring of GKM spaces}\label{GKM section}
 Let $(M,\omega,\T,\psi)$ be a Hamiltonian $\T$-space (so $M^\T$ is discrete) and assume that $\dim(\T)\geq 2$.
 Then the $\T$-action is called 
 {\bf GKM (Goresky-Kottwitz-MacPherson \cite{GKM})}, or equivalently $(M,\omega,\T,\psi)$ is called a {\bf GKM space}, if for every codimension 
 one subtorus $K\subset \T$, the submanifold fixed by $K$ has dimension at most $2$. It can be checked that this condition is equivalent to requiring that for every fixed point $p\in M^\T$, the weights of the isotropy action at $p$, $w_{1},\ldots,w_n\in \ell^* \subset \mathfrak{t}^*$, are pairwise linearly independent. 
 Let $\mathfrak{k}_i\subset \mathfrak{t}$ be $\ker(w_i)$, and $K_i=\exp(\mathfrak{k}_i)$. From the definition of GKM space it follows that for each $i=1,\ldots,n$, the connected component of $M^{K_i}$ containing $p$ is a $2$-sphere, called {\bf isotropy sphere}. The circle group $\T/K_i$ acts effectively on such sphere, and this action has two fixed points, one of them being $p$. The combinatorics of the arrangement of isotropy spheres, together with the information on their stabilizers, is encoded in a labeled graph, called {\bf GKM graph $\Gamma=(V,E)$}:
 \begin{itemize}
 \item The vertex set $V$ coincides with the fixed point set $M^\T$.
 \item Given distinct $p,q\in V$, there exists a directed edge $e=\overrightarrow{p\,q}$ from $p$ to $q$ if and only if there exists a $2$-sphere $S^2$ fixed by some codimension one subgroup $K\subset \T$ such that the fixed points of the action of the quotient circle $\T/K$ on $S^2$ are precisely $p$ and $q$; we refer to this sphere as the {\bf sphere associated to (the edge) $e$}.
 \item Every edge $e=\overrightarrow{p\,q}\in E$ is labeled by a weight $w(e)\in \ell^*$, defined as the weight of the isotropy $\T$-action on $T_qS^2$, where $S^2$ is the sphere associated to $e=\overrightarrow{p\,q}$.
 \end{itemize}
 Every time $p$ is connected to $q$ by an edge $e=\overrightarrow{p\,q}$ (with weight $w(e)$), then by definition also $q$ is connected to $p$ by an edge $-e=\overrightarrow{q\,p}$ (with weight $w(-e)=-w(e)$). In order to avoid having two edges representing geometrically the same sphere, we will choose one of those edges by picking an \emph{orientation} on the edge set $E$ in the following way. 
Pick a generic $\overline{\xi}$ and let $\mu\colon M\to \R$ be the $\overline{\xi}$-component of the moment map $\psi$, as defined before. Each isotropy sphere is a symplectic submanifold with an effective Hamiltonian action of a circle with two fixed points $p$ and $q$. Since $\overline{\xi}$ is generic we have
 $w(\overline{\xi})\neq 0$, for all $w\in W_p$. This implies that $\mu(p)\neq \mu(q)$, and so for each isotropy sphere we choose the directed edge $e=\overrightarrow{p\,q}$ such that $\mu(p)<\mu(q)$. 
 We will refer to this graph as the {\bf oriented GKM graph} (associated to $(M,\omega,\T,\psi,\overline{\xi})$) and denote it by $\Gamma^o=(V,E^o)$.
 We also define an {\bf increasing path} $\gamma$ from $p$ to $q$ in the oriented GKM graph $\Gamma^o$, where $p,q\in V$,
 to be an ordered sequence of edges in $E^o$ of the form  $(\overrightarrow{p\,p_1},\,\overrightarrow{p_1\,p_2},\ldots,\,\overrightarrow{p_j\,q})$.
 Observe that if $p\succeq q$ then the set of increasing paths from $p$ to $q$ is empty.

 For a GKM space $(M,\omega,\T,\psi,\overline{\xi})$ with oriented GKM graph $\Gamma^o=(V,E^o)$, the weights in the
 negative normal bundle at $p$, namely those in $W_p^+$, coincide with the weights associated to the edges $e_j\in E^o$ of the form $\overrightarrow{q_j\,p}$. Thus 
 the equivariant (K-theoretical) Euler class of the negative bundle at $p\in M^\T$ in Definition \ref{eec} can be expressed as
 \begin{equation*}
 \ee(p)=\prod_{e_j}(1-e^{2\pi \il w(e_j)})
 \end{equation*}
 where the product is over all the $\lambda_p$ edges $e_j\in E^o$ ending at $p$, i.e. $e_j=\overrightarrow{q_j\,p}\in E^o$, for some $q_j\in V$.
 
In analogy with the equivariant cohomology ring, the GKM graph determines which elements $f\in \Ke(M^\T)$
come from classes in $\Ke(M)$. 
This is proved by Knutson and Rosu in the Appendix of \cite{R}. 
 \begin{thm}[Knutson, Rosu '03]\label{GKM k}
 Let $(M,\omega,\T,\psi)$ be a GKM space, and $\Gamma=(V,E)$ the associated GKM graph. Then $\tau \in \Ke(M^\T)\simeq \bigoplus_{p\in M^\T}R(\T)$ is an element of $\Ke(M)$ if and only if for every $e=\overrightarrow{p\,q}\in E$ 
 \begin{equation}\label{k cdt}
 \tau(p)-\tau (q)=f\,(1-e^{2\pi \il w(e)})\quad \mbox{for some}\quad f\in R(\T)\,.
 \end{equation}
 \end{thm}
 Observe that the elements $\tau\in \Ke(M^\T)$ satisfying \eqref{k cdt} indeed form a ring. Moreover condition \eqref{k cdt} is equivalent to requiring $$\tau(p)-\tau (q)=\tilde{f}\,(1-e^{- 2\pi \il w(e)}) \quad \mbox{for some} \quad \tilde{f}\in R(\T),$$
 thus it is sufficient to check condition \eqref{k cdt} on the edges of the oriented GKM graph.
 \subsection{Symplectic toric manifolds as GKM spaces}\label{toric GKM}
 In the following sections we will focus on symplectic toric manifolds, i.e. Hamiltonian $\T$-spaces $(M,\omega,\T,\psi)$ with $\dim(\T)=\displaystyle\frac{\dim(M)}{2}$. Observe that for every $p\in M^\T$, the isotropy weights at $p$ form a $\Z$-basis of $\ell^*$, the weight lattice of $\mathfrak{t}^*$, hence they are
 pairwise linearly independent. Thus
 symplectic toric manifolds are a special class of 
 GKM spaces. Moreover, their oriented GKM graph can be recovered from the one-skeleton of the image of the moment map $\psi(M)$:
 \begin{itemize}
 \item The vertices of the GKM graph are the vertices of the polytope.
 \item The oriented edges of the GKM graph are precisely the edges of the polytope, oriented by using $\mu$.
 \item  The weight labeling the edge $e=\overrightarrow{p\,q}$ is precisely the primitive element  $w\in \ell^*$ such that 
 $\psi(q)-\psi(p)=m\,w$, for some $m\in \R_{>0}$. 
 \end{itemize}
\subsection{Definition of the Local Index}\label{section local index}
Let $(M,\omega,\T,\psi)$ be a symplectic toric manifold. 
In this subsection, for every $q\in M^\T$, we construct a map
$$
\ii_q\colon \K \to R(\T)
$$
called the local index map at $q$.
Before defining the local index map, we set some notation and easy facts. 

Recall that the {\bf equivariant K-theory push-forward map} $\textrm{Ind}\colon \K \to R(\T)$, also called the {\bf (equivariant) index homomorphism}, can be computed by using the Atiyah-Segal formula \cite{AS} in the following way:
\begin{equation}\label{AS index}
\textrm{Ind}(\tau)=\sum_{p\in M^\T}\frac{\tau(p)}{\displaystyle\prod_{w_j\in W_p}(1-e^{2\pi \il w_j})}\,.
\end{equation}
Note that this is not in general a homomorphism of rings, but only a homomorphism of $R(\T)$-modules.

For the trivial bundle $\mathbf{1}\in \Ke(M)$ one has: 
$$
\textrm{Ind}(\mathbf{1})=\sum_{p\in M^\T}\frac{1}{\displaystyle\prod_{w_j\in W_p}(1-e^{2\pi \il w_j})}.
$$
For any generic $\overline{\xi}\in \mathfrak{t}$, consider the restriction map $r_{\bar{\xi}}\colon \mathcal{K}_\T(\p)\to \mathcal{K}_{S^1}(\p)$, where
$S^1=\exp(t\,\overline{\xi})$, and let $\mu\colon M\to \R$ be the associated $\overline{\xi}$-component of the moment map.
Then by \cite[Corollary 2.7]{Ha} we have that $r_{\bar{\xi}}(\ii(\mathbf{1}))$ is the number of
fixed points with no negative weights w.r.t.\ $\mu$, which is always $1$ as $M$ is connected and the action Hamiltonian.
Since $\overline{\xi}$ was generic, we obtain
\begin{equation}\label{index 1-2}
\ii(\mathbf{1})=1.
\end{equation}

Let $w_1,\ldots,w_{\lambda_q}$
be the weights in the negative normal bundle of $\mu$ at $q\in M^\T$.  Define $\mathcal{H}_q$ to be $\lambda_q$-dimensional affine subspace of $\R^n$ given by $\psi(q)+\R\langle w_1,\ldots,w_{\lambda_q}\rangle$.
 It is easy to see that $\mathcal{H}_q\cap \psi(M)$ is a $\lambda_q$-dimensional face of the polytope $\psi(M)$. Moreover 
 $$H_q:=\psi^{-1}(\mathcal{H}_q\cap \psi(M))$$
 is a closed $\T$-invariant symplectic submanifold of $M$, 
  which can be thought of as the ``{\bf flow-down} at $q$". We denote by $s_0=q,s_1,\ldots,s_j$ the $\T$-fixed points in $H_q$.  
 Similarly, define $F_q$ to be 
 \begin{equation*}\label{def fq}
 F_q=\psi^{-1}(\mathcal{F}_q\cap \psi(M)),
 \end{equation*}
 where $\mathcal{F}_q$ is the $(n-\lambda_q)$-dimensional affine space given by 
 $\psi(q)+\R\langle w_{\lambda_q+1},\ldots,w_{\lambda_n}\rangle$.   
 Note that $F_q$ is also a closed $\T$-invariant symplectic submanifold of $M$, 
  which can be thought as the ``{\bf flow-up} at $q$". Therefore, \emph{mutatis mutandis}, what is claimed below for $H_q$ also holds for $F_q$.
  
 The action of $\T$ on $H_q$ is clearly not effective. In fact the subtorus $$\T_q^0=\exp(\{\xi\in \g\mid w_i(\xi)=0,\;i=1,\ldots,\lambda_q\})\subset \T$$ acts trivially on an open neighborhood of $q$ in $H_q$, and hence it
acts trivially on $H_q$. Thus $H_q$ is acted on by the quotient torus $\T_q=\T/\T^0_q$ 
 and the Lie algebra $Lie(\T_q)^*$ can be identified with $\R\langle w_1,\ldots,w_{\lambda_q}\rangle$. Moreover the action of $\T_q$
on $H_q$ is effective. Indeed, every point $s\in H_q$ fixed by $\T_q$ is also a $\T$-fixed point in $M$, and the weights $w'_1,\ldots,w'_{\lambda_q}$ of the
isotropy $\T_q$ action on $T_sH_q$ correspond to those weights $w'_1,\ldots,w'_n$ of the isotropy action of $\T$ on $T_sM$ which belong to $\R\langle w_1,\ldots,w_{\lambda_q}\rangle$.
Moreover, since the $\T$-action on $M$ is toric, we have that $\Z\langle w'_1,\ldots,w'_n\rangle=\ell^*$, which implies that
$\Z\langle w'_1,\ldots,w'_{\lambda_q}\rangle=\ell^*_q$, where $\ell^*_q=\ell^*\cap \R\langle w_1,\ldots,w_{\lambda_q}\rangle$. 
Thus the action is effective.
We conclude that $H_q$
is endowed with a symplectic toric action of $\T_q$. 

Let $S^2$ be the $2$-sphere
endowed with the standard symplectic form and a symplectic toric $S^1$ action rotating the sphere $S^2$, with speed $1$, keeping the north and the south poles fixed. Consider a moment map, $h\colon S^2\to \R$, such that $0=h(S)<h(N)=1$, where
$N$ and $S$ denote the north and the south poles respectively. 
Let $H_q\times S^2$ be the symplectic manifold with symplectic form given by the sum of the pull-backs of the symplectic forms on 
$H_q$ and $S^2$. This manifold is endowed with the (non effective) Hamiltonian action of
$\T\times S^1$, i.e. $(t,t')\ast (s,s')=(t\ast  s, t'\ast s' )$ for every $(t,t')\in \T\times S^1$ and $(s,s')\in H_q\times S^2$. The $\T\times S^1$ fixed points in $H_q\times S^2$ are given by $$H_q^{\T}\times \{S\}=\{q_0=(s_0,S),\ldots,q_j=(s_j,S)\}$$
and $$H_q^{\T}\times \{N\}=\{q'_0=(s_0,N),\ldots,q'_j=(s_j,N)\},$$ for all $i=0,\ldots,j$, where $q_0=(s_0,S)=(q,S)$.

Note that the splitting $\T\times S^1$ allows us to regard $\mathfrak{t}^*$ as a subspace of $Lie(\T\times S^1)^*$. With abuse of notation we consider the weights of the $\T$-action as elements in $Lie(\T\times S^1)^*$.
Let $-w_0\in Lie(\T\times S^1)^*$ be the weight of the isotropy 
$\T\times S^1$ action at $T_{q}(\{q_0\}\times S^2)\subset T_{q_0}\,(H_q\times S^2)$ (see Figure \ref{figure cutting}). 

\begin{figure}[htbp]
\begin{center}
\psfrag{q0}{\footnotesize{$q_0$}}
\psfrag{q1}{\footnotesize{$q_1$}}
\psfrag{q2}{\footnotesize{$q_2$}}
\psfrag{q0p}{\footnotesize{$q_0'$}}
\psfrag{q1p}{\footnotesize{$q_1'$}}
\psfrag{q2p}{\footnotesize{$q_2'$}}
\psfrag{w0}{\footnotesize{$w_0$}}
\psfrag{w1}{\footnotesize{$w_1$}}
\psfrag{w2}{\footnotesize{$w_2$}}
\psfrag{w20}{\footnotesize{$w_2+w_0$}}
\psfrag{w10}{\footnotesize{$w_1+w_0$}}
\psfrag{w12}{\footnotesize{$w_1-w_2$}}
\psfrag{p0}{\footnotesize{$p_0$}}
\psfrag{p1}{\footnotesize{$p_1$}}
\psfrag{p2}{\footnotesize{$p_2$}}
\psfrag{s0}{\footnotesize{$s_0$}}
\psfrag{s1}{\footnotesize{$s_1$}}
\psfrag{s2}{\footnotesize{$s_2$}}
\includegraphics[width=12cm]{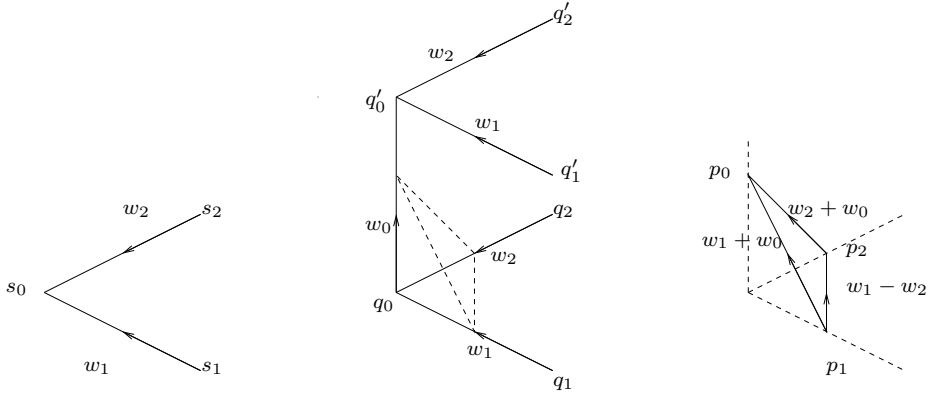}
\caption{Local pictures for $H_q$, $H_q \times S^2$ and $\widetilde{H}_q^{\epsilon}$.}
\label{figure cutting}
\end{center}
\end{figure}
Around $q_0=(q,0)$, the moment map of the Hamiltonian $\T\times S^1$ action on
$H_q\times S^2$ can be written as
$$
\psi'(z_1,\ldots,z_{\lambda_q},z)=-w_0 \frac 1 2 |z|^2+w_1 \frac 1 2 |z_1|^2+\ldots+w_{\lambda_q} \frac 1 2 |z_{\lambda_q}|^2+\psi(q)\;,
$$
where $(z_1,\ldots,z_{\lambda_q},z)$ are complex coordinates on $H_q\times S^2$ around $q_0$.
Since the $\T$-action on $M$ is toric, the weights of the isotropy $\T$-action at
$q$, given by $w_1,\ldots,w_n$, form a $\Z$-basis of $\ell^*$, the weight lattice of $\mathfrak{t}^*$; hence
$w_0,w_1,\ldots,w_n$ is a $\Z$-basis of the dual lattice of $\T\times S^1$.
Let $\xi_q$ be the element of $Lie(\T\times S^1)$ such that $w_0(\xi_q)=-1$,
$w_i(\xi_q)=1$ for $i=1,\ldots,\lambda_q$, and $w_i(\xi_q)=0$ for $i=\lambda_{q}+1,\ldots,n$.
Then $C_q=\exp(\xi_q)\subset \T\times S^1$ is a circle acting effectively on $H_q\times S^2$
with moment map $\varphi$ which around $q_0$ is given by
$$
\varphi(z_1,\ldots,z_{\lambda_q},z)= \frac 1 2 |z|^2+\frac 1 2 |z_1|^2+\cdots+ \frac 1 2 |z_{\lambda_q}|^2+\psi(q)(\xi_q)\,.
$$

Let $\widetilde{H}_q^{\epsilon}$ be the symplectic reduction of $H_q\times S^2$ at $\epsilon+\psi(q)(\xi_q)$ with respect to the action of $C_q$, for $\epsilon>0$ sufficiently small, i.e. 
$$\widetilde{H}_q^{\epsilon}=\varphi^{-1}(\epsilon+\psi(q)(\xi_q))/C_q.$$ Observe that $\widetilde{H}_q^{\epsilon}$ is symplectomorphic to the complex projective space $\C \pp^{\lambda_q}$, 
endowed with the (Hamiltonian) residual action of $\widetilde{\T}_{q}=(\T\times S^1)/C_q$.
Denote its fixed points by $p_0,p_1,\ldots,p_{\lambda_q}$ with 
$$p_0 = \left(\, \varphi^{-1}(\epsilon+\psi(q)(\xi_q)) \cap (\{q\}\times S^2)\, \right)/C_q,$$ 
$$p_j = \left(\,\varphi^{-1}(\epsilon+\psi(q)(\xi_q)) \cap (\psi')^{-1}(\overrightarrow{q_j\,q_0})\,\right)/C_q$$
(see Figure \ref{figure cutting}).
Moreover, the tuples of isotropy weights of the
$\widetilde{\T}_{q}$ action at the fixed points of $\widetilde{H}_q^{\epsilon}$ are given by: $$W_{p_0}=\{w_0+w_1,\ldots,w_0+w_{\lambda_q}\}$$
and $$W_{p_i}=\{-(w_i+w_0),w_1-w_i,\ldots,w_{i-1}-w_i,w_{i+1}-w_i,\ldots,w_{\lambda_q}-w_i\}$$ for all $i=1,\ldots,\lambda_q$.
 
We are ready to define the Local Index map of a class $\tau\in \K$ at $q\in M^\T$, and divide its computation into the following steps:
$\;$\\

{\bf (i)} \emph{Mapping $\tau\in \K$ to $\mathcal{K}_{\T\times S^1}(H_q\times S^2):$}\\
Consider the following commuting diagram
 $$
 \xymatrix{ 
\Ke(M) \ar[d] \ar[r]^-{r_1}\ar[drr]^-{\;\;\;\;\;\;\;\;\widetilde{i}^*\circ \widetilde{r}} &  \Ke(H_q)\ar@{.>}[d] \ar[r]^-{r_2}  & \mathcal{K}_{\T\times S^1}(H_q\times S^2)\ar[d]_{\widetilde{i}^*}\\
\Ke(M^\T) \ar[r] &  \Ke(H_q^\T) \ar[r]  & \mathcal{K}_{\T\times S^1}((H_q\times S^2)^{\T\times S^1}) \\
}
$$
where the maps involved are defined as follows.
The map $r_1$ is the restriction induced by the $\T$-equivariant inclusion
$H_q\hookrightarrow M$. To define the map $r_2 \colon  \Ke(H_q) \to \mathcal{K}_{\T\times S^1}(H_q\times S^2)$ note that, as $S^1$ acts trivially on $H_q$ and $\T$ acts trivially on $S^2$, we have a canonical identification
$$ \mathcal{K}_{\T\times S^1}(H_q\times S^2) \cong \mathcal{K}_{\T}(H_q)\otimes \mathcal{K}_{S^1}(S^2).$$
Define $r_2$ as tensoring with $1 \in \mathcal{K}_{S^1}(S^2)$, i.e. $r_2( r_1 (\tau))$ is the image of $r_1 (\tau) \otimes 1$ in $\mathcal{K}_{\T\times S^1}(H_q\times S^2)$ under the above identification. 
The vertical arrows are the restrictions of the equivariant K-theory rings to those of the fixed point sets; in particular
$\widetilde{i}\colon (H_q\times S^2)^{\T\times S^1}\to H_q\times S^2$. Let $\widetilde{r}$ be $r_2\circ r_1$. 

In Step {\bf (i)} we map
$\tau\in \K$ to $\widetilde{r}(\tau)\in \mathcal{K}_{\T\times S^1}(H_q\times S^2)$. In practice, we work with 
$\widetilde{i}^*\circ \widetilde{r}(\tau)\in \mathcal{K}_{\T\times S^1}\big((H_q\times S^2)^{\T\times S^1}\big)$.
An explicit procedure is given as follows:
\begin{itemize}
\item determine the restrictions of $\tau\in \K$ to $H_q^\T=\{s_0=q,\ldots,s_j\}$;
\item then $\widetilde{i}^*\circ\widetilde{r}(\tau)$ is $\tau(s_l)$ at $q_l$ and $q'_l$, for all $l=0,\ldots,j$. 
\end{itemize}
Note that these restrictions
live in $R(\T)$, which we regard as a subring of $R(\T\times S^1)$ using the identifications $R(\T)=\Z[\e{w_1},\ldots,\e{w_n},\e{(-w_1-\ldots -w_n)}]$
and $R(\T\times S^1)=\Z[\e{w_0},\e{w_1},\ldots,\e{w_n},\e{(-w_1-\ldots -w_n-w_0)}]$ (see \eqref{id R(T)}).

{\bf (ii)} \emph{Mapping $\widetilde{r}(\tau) \in \mathcal{K}_{\T\times S^1}(H_q\times S^2)$ to $\mathcal{K}_{\widetilde{\T}_{q}}(\widetilde{H}^{\epsilon}_q)$:} 
\begin{equation}\label{k map}
\kappa\colon \mathcal{K}_{\T\times S^1}(H_q\times S^2)\to \mathcal{K}_{\widetilde{\T}_{q}}(\widetilde{H}^{\epsilon}_q)
\end{equation}
In practice, given $\widetilde{r}(\tau)\in\mathcal{K}_{\T\times S^1}(H_q\times S^2)$ we compute
the image of $\kappa(\widetilde{r}(\tau))$ in $\mathcal{K}_{\widetilde{\T}_{q}}((\widetilde{H}^{\epsilon}_q)^{\widetilde{\T}_{q}})$. That is, we find the restrictions of $\kappa (\widetilde{r}(\tau) )$ to each of the $\lambda_q+1$ fixed points $p_0,p_1,\ldots,p_{\lambda_q}$ of $\widetilde{H}_q^{\epsilon}\cong \C \pp^{\lambda_q}$ using the recipe below.
The value $\widetilde{r}(\tau) (q_0)$ (equal to the value $\tau(s_0)$) is an element of $R(\T)$, which we regard as a subring of $R(\T\times S^1)$. 
Since the weights
$w_1,\ldots,w_n$ form a $\Z$-basis of $\ell^*$, we can express
$\tau(q_0)$ as $f(\e{w_1},\ldots,\e{w_n},\e{(-w_1-\ldots-w_n)})$, where $f(X_1,\ldots,X_{n+1})$ is a polynomial with integer
coefficients. Let 
$$f_0=f(\e{(w_0+w_1)},\ldots,\e{(w_0+w_{\lambda_q})},\e{w_{\lambda_q+1}},\ldots,\e{w_n},\e{(-w_1-\ldots-w_n-\lambda_qw_0)})$$
and 
$$
f_j=f(\e{(w_1-w_j)},\ldots,\e{(w_{\lambda_q}-w_j)},\e{w_{\lambda_q+1}},\ldots,\e{w_n},\e{(-w_1-\ldots-w_n+\lambda_qw_j)})\,,
$$
for all $j=1,\ldots,\lambda_q$ 
(i.e.\ the $j$-th argument of $f_j$ is $e^0=1$). Define the restriction of $\kappa(\tau)$ to the fixed point set to be $f_j$ at the point $p_j$, for all $j=0,\ldots,\lambda_q$.
Observe that this class does represent indeed an element of $\mathcal{K}_{\widetilde{\T}_q}(\widetilde{H}^{\epsilon}_q)$. By 
Theorem \ref{GKM k} it is sufficient to check that
$f_0-f_j\equiv 0$ mod $(1-\e{(w_0+w_j)})$ and $f_i-f_j\equiv 0$ mod $(1-\e{(w_i-w_j)})$, i.e. one needs to check that $f_0$ is equal to $f_j$ when setting $w_0+w_j=0$
and $f_i$ is equal to $f_j$ when setting $w_i-w_j=0$, which follows easily from the definition
of the $f_i$'s.
\begin{rmk}\label{remark kirwan map}
It is worth observing (though we are not going to use it) that the above map $\kappa$ is the Kirwan map relating the K-theory rings of a manifold and of its symplectic reduction. We devote the Appendix to a careful description of the Kirwan map. \end{rmk}
\begin{defin}\label{definition local index}
 Let $q\in M^\T$ and $\tau\in \Ke(M)$. The {\bf local index} of $\tau$ at $q$, denoted by $\ii_q(\tau)$, is the element of $R(\T)$ defined as
 $$
 \ii_q(\tau):= \alpha_q \circ Ind ((\kappa\circ \widetilde{r})(\tau))\,,
 $$
 where $Ind\colon \mathcal{K}_{\widetilde{\T}_q}(\widetilde{H}_q^{\epsilon})\to R(\widetilde{\T}_q)$ is the index homomorphism, and 
 $\alpha_q\colon  R(\widetilde{\T}_q)=\mathcal{K}_{\widetilde{\T}_q}(\textrm{pt}) \rightarrow \mathcal{K}_{\T}(\textrm{pt})=R(\T)$ is the homomorphism sending $w_0$ to $0$ and $w_j$ to itself, for all $j\neq 0$. The {\bf local index map} at $q\in M^\T$ is the map 
\begin{equation}\label{index map eq}
\ii_q\colon \Ke(M)\to R(\T)\,
\end{equation}
that assigns $\ii_q(\tau)$ to $\tau$.
\end{defin}
Note that the only information needed for computing $\ii_q(\tau)$
are the weights of $\T$-action on $T_q H_q$ and the value of $\tau$ at $q$. 
Moreover the computation is relatively easy thanks to the combinatorial recipe for computing the restriction of $(\kappa\circ \widetilde{r})(\tau)\in \mathcal{K}_{\widetilde{\T}_q}(\widetilde{H}_q^{\epsilon})$ to the $\widetilde{\T}_q$ fixed point set given above, and thanks to the Atiyah-Segal formula \eqref{AS index}.

\begin{rmk}\label{comparison}
The above definition of the local index is inspired by, though different from, the one in \cite{GK}. 
 In our definition of local index, we make an explicit choice of the circle $C_q$ used for obtaining $\widetilde{H}_q^{\epsilon}$ through symplectic reduction. The choice  is different for each $q\in M^\T$. With this choice the reduced space is smooth and one can explicitly calculate the local index using the algorithm we provide. 
 \end{rmk}
\begin{exm}
As an example we calculate the local index of the second class in Figure \ref{fig:hirzebruchclasses} at the point $q$ where its value is 
$(1-e^{2 \pi i (x-y)})e^{2\pi i y}$. Call this class $\tau$. 
\begin{figure}[htbp]
\begin{center}
\includegraphics[width=12cm]{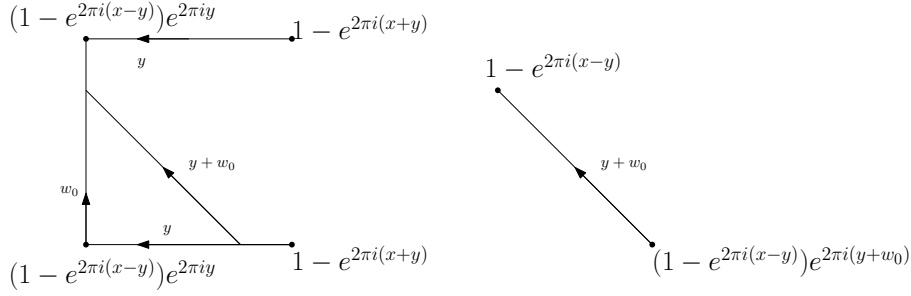}
\caption{The values of the class $\widetilde{r}(\tau)$ at the fixed points of $H_q \times S^2$ close to $(q,0)$ and the values of the class $\kappa(\tau)$ at the fixed points of $\widetilde{H}_q^{\epsilon}$.}
\label{fig:local index}
\end{center}
\end{figure}
Following the notation from the definition of the local index, we have
\begin{align*}
\tau(q)&=f(e^{2\pi i y},e^{2 \pi i (x-y)})=(1-e^{2 \pi i (x-y)})e^{2\pi i y},\\
\kappa(\tau)(p_0)&=f_0=f(e^{2\pi i (y+w_0)},e^{2 \pi i (x-y)})=(1-e^{2 \pi i (x-y)})e^{2\pi i (y+w_0)},\\
\kappa(\tau)(p_1)&=f_1=f(1,e^{2 \pi i (x-y)})=1-e^{2 \pi i (x-y)}.
\end{align*}
The values of the class $\widetilde{r}(\tau)$ at the fixed points of $H_q \times S^2$ close to $(q,0)$ are presented on the left side of Figure \ref{fig:local index}, while the picture on the right represents the values of the class $\kappa(\tau)$ (c.f.\ the Appendix). 
The local index of $\tau$ at $q$ is equal to 
\begin{align*}
\ii_q(\tau)&=\frac{(1-e^{2 \pi i (x-y)})e^{2\pi i (y+w_0)}}{1-e^{2\pi i (y+w_0)}}+\frac{1-e^{2 \pi i (x-y)}}{1-e^{-2\pi i (y+w_0)}}\\
&=-\frac{1-e^{2 \pi i (x-y)}}{1-e^{-2\pi i (y+w_0)}}+\frac{1-e^{2 \pi i (x-y)}}{1-e^{-2\pi i (y+w_0)}}=0.\\
\end{align*}
\end{exm}

\begin{prop}\label{prop indexproperties}
The local index satisfies the following properties:
\begin{itemize}
\item[(i)] (Additivity) for any $\tau, \tau' \in \Ke(M)$, and any $q\in M^\T$ have $$\ii_q(\tau +\tau')=\ii_q(\tau) + \ii_q(\tau').$$
\item[(ii)] If $\tau(q)=f \, \ee(q)$ for some $ f \in R(\T)$ 
then $$\ii_q(\tau)=f=\tau(q)/\ee(q).$$
In particular $\ii_q(\tau)=0$ if $\tau(q) = 0$. \\
\item[(iii)] If $\tau'(q)=0$ then $\ii_q(\tau + f \, \tau')=\ii_q (\tau)$, for every $f \in R(\T)$,
\end{itemize}
\end{prop}
The condition in (ii) is satisfied when $\tau(q')=0$ for all $q' \prec q$, see the proof of Proposition \ref{existence kirwan}.
Note that it is \emph{not} always true that $\ii_q(f \, \tau)=f \, \ii_q(\tau)$, where $f\in R(\T)$ and $\tau\in \Ke(M)$. In other words, the local index map $\ii_q\colon \Ke(M)\to R(\T)$
is not a morphism of $R(\T)$-modules. This is also true for the local index defined in \cite{GK} (see \cite[Remark 5.1]{GK}).

\begin{proof}[Proof of Proposition \ref{prop indexproperties}]
 The first property follows from additivity of the index homomorphism. 
 To prove the second property, present $f$ as a polynomial in $\e{w_j}$ where $w_1,\ldots,w_n$ are the weights of $\T$-action on $T_q M$,
 with $w_1,\ldots,w_{\lambda_q}$ being the weights of $\T$-action on $T_q H_q$, i.e. the weights in $ W_q^+$.
 Let $f_0$ be the polynomial obtained from $f$ by substituting $w_j$ with $w_j+w_0$ for $j=1,\ldots, \lambda_q$.  
Then the values of $\kappa(\widetilde{r}(f \tau))$ at the fixed points $p_0,\ldots,p_{\lambda_q}$ of $\widetilde{H}^{\epsilon}_q$ are
$$f_0 \,\prod_{j=1}^{\lambda_q}(1-\e{(w_j+w_0)}), 0,\ldots, 0.$$
Therefore, from \eqref{AS index} it follows that the index of $\kappa(\widetilde{r}(f \tau))$ is $f_0$,
and the local index $\ii_q(f \tau)$ is the image of $f_0$ under the homomorphism  $ \mathcal{K}_{\widetilde{\T}_q}(\textrm{pt}) \rightarrow \mathcal{K}_{\T}(\textrm{pt})$ sending $w_0$ to $0$, so it is $f$.
The last property follows from the first two.
\end{proof}

 Recall that by $F_p$ we denote the flow-up manifold at $p\in M^\T$, as defined on page~\pageref{def fq}. We can now define {\em i-canonical classes}. 
 \begin{defin}\label{definition canonical classes} 
Let $(M,\omega,\T,\psi)$ be a symplectic toric manifold of dimension $2n$, together with 
a choice of a generic component of the moment map $\mu=\psi^{\overline{\xi}} \colon M\to \R$. Then for each $p\in M^\T$,
a Kirwan class $\tau_p \in \K$ satisfying the following properties:
\begin{enumerate}
 \item $\ii_q(\tau_p)=1$ for all points $q \in F_p\cap M^\T$;
\item $\ii_q(\tau_p)=0$ for all points $q \notin F_p\cap M^\T$;
\end{enumerate}
is called an {\bf i-canonical class} at (the fixed point) $p$.
 \end{defin}
\begin{rmk}\label{class 1 rmk}
 Note that the equivariant K-theory class of the trivial bundle $\mathbf{1}$ is an i-canonical class. Indeed, if $p_{min}\in M^\T$ is the fixed point at which $\mu$ attains its minimum, then $F_{p_{min}}$ is the whole manifold $M$ and for any $q \in M^\T$ we have that $\ii_q(\mathbf{1})=\ii (\mathbf{1}_{\C\pp^{\lambda_q}})=1$ (see \eqref{index 1-2}).
 \end{rmk}

Observe that as the i-canonical classes are a special set of Kirwan classes, we have that: 
\begin{enumerate}
\item[$\bullet$] if for each $p\in M^\T$ there exists an i-canonical class $\tau_p$, then the set $\{\tau_p\}_{p\in M^\T}$ is a basis for $\Ke(M)$ as a module over $R(\T)$;
\item[$\bullet$] $\tau_p(p)=\ee(p)$;
\item[$\bullet$] $\tau_p(q)=0$ for all $q \in M^\T \setminus \{p\}$ with $q \prec p$.
\end{enumerate}
We will see later that a stronger condition is true: $\tau_p(q)=0$ for all $q \in M^\T \setminus V_p^+$, where $V_p^+$ is defined in Definition \ref{def vpplus}.

As we remarked in Section \ref{sec:1}, Kirwan classes are never unique, unless $M$ is a point. 
Conditions (1) and (2) in Definition \ref{definition canonical classes} guarantee that
if i-canonical classes exist then they are unique, hence their name:
they are \emph{canonically associated} to $(M,\omega,\T,\psi,\overline{\xi})$.
The ``i'' refers to ``index", as they are defined using the notion of local index.
\begin{prop}[Uniqueness of canonical classes]\label{uniqueness}
 If an i-canonical class $\tau_p$ exists then it is unique.
  \end{prop}
\begin{proof}
 Suppose that there exist two K-theory classes $\tau_p$ and $\tau'_p$ satisfying all conditions of Definition \ref{definition canonical classes}, for some $p\in M^\T$.
 Then the class $\eta=\tau_p-\tau'_p$ is supported on points leaving above $p$ with respect to the order $\prec$ defined in \eqref{ordering fixed points}, and its local index is zero at all fixed points.
  If $\eta$ were nonzero then, by injectivity of \eqref{inj 2}, there would exist a fixed point $q \succ p$ with $\eta(q) \neq 0$.
 Take minimal such $q$ (minimal with respect to $\prec$). 
 Then, if $q_1,\ldots,q_{\lambda_q}$ are the fixed points in $H_q$ connected to $q$ through an edge of the (oriented) GKM graph
 $\Gamma$, we have that $q_j\prec q$, implying $\eta(q_j)=0$ for $j=1,\ldots, \lambda_q$. So Theorem \ref{GKM k}
 implies that $\eta(q)$ must be a (nonzero) multiple of $\ee(q)= \prod_{w_j\in W_q^+}(1-e^{2\pi \il w_j}) $. 
 But condition (ii) of Proposition \ref{prop indexproperties} would then imply that $\ii_q(\eta) \neq 0$. This contradiction proves that $\eta$ must be the zero class and
 hence $\tau_p=\tau'_p$.
\end{proof}
Section \ref{section construction} is devoted to proving that, indeed, for symplectic toric manifolds i-canonical classes always exist. 
\section{Construction}\label{section construction}

The main goal of this section is to prove the following 
\begin{prop}[Existence of i-canonical classes]\label{existence canonical}
Let $(M,\omega,\T,\psi)$ be a symplectic toric manifold of dimension $2n$, together with 
a choice of a generic component of the moment map $\mu=\psi^{\overline{\xi}} \colon M\to \R$. 
Then for each $p\in M^\T$, there exists an i-canonical class $\tau_p$.
\end{prop}
We will prove this by explicitly exhibiting a set of K-theory classes satisfying properties (1) and (2) in Definition \ref{definition canonical classes}.
The proof of Proposition \ref{existence canonical} is divided into two parts: the index increasing case (Proposition  \ref{classes from cohomology are canonical}) and non-index increasing case (Subsection \ref{toric non-increasing}). 
\begin{defin}\label{ii and nii}
Let $(M,\omega,\T,\psi,\overline{\xi})$ be a GKM space with oriented GKM graph $\Gamma^o=(V,E^o)$. 
Then the space is called {\bf index increasing} if for every edge $e=\overrightarrow{p\,q}\in E^o$ we have $\lambda_p<\lambda_q$, and {\bf non-index increasing} otherwise.
\end{defin}
The Hirzebruch surface in Figure \ref{fig:hirzebruchclasses} is an example of non-index increasing GKM space.

For toric manifolds there is a natural algorithm for constructing a basis of the equivariant K-theory ring consisting of special Kirwan classes
which are equivariant Poincar\'e duals to the flow up submanifolds.
As we will prove in Section \ref{toric increasing}, 
in the index increasing case these equivariant K-theory classes are indeed i-canonical classes 
(see Proposition \ref{classes from cohomology are canonical}). In the non-index increasing case we will need to modify them to make them ``canonical'', as explained
in Section \ref{toric non-increasing}.
Below we recall this algorithm. 

Let  $(M^{2n},\omega,\T, \psi)$ be a symplectic toric manifold, and let $\mu=\psi^{\overline{\xi}}\colon M\to \R$ be a generic component of the moment map.  Let $\Gamma=(V,E)$ (resp.\ $\Gamma^o=(V,E^o)$) be the associated GKM graph (resp.\ oriented GKM graph). This GKM space is not necessarily index increasing. 
We recall that for every $p\in M^\T$ the flow-up at $p$, denoted by $F_p$, is a $\T$-invariant submanifold of $M$ (see the discussion on page~\pageref{def fq}, Section \ref{section local index}). 
Thus the normal bundle to $F_p$, which we denote by $N(F_p)$, is $\T$-invariant, and for each $q\in M^\T\cap F_p$ the set of weights of the $\T$-representation on $N(F_p)|_q$ is given by $\{w(\overrightarrow{r\,q})\}$, where $r\in M^\T\setminus F_p$ and $\overrightarrow{r\,q}\in E$; note that such an edge does not necessarily belong to $E^o$. 
We have that the (K-theoretical) equivariant Euler class of the normal bundle $N(F_p)|_q$ is
\begin{equation}\label{eeKc normal}
\mathsf{e}_\T\big( N(F_p)|_q\big)=\prod_{
\begin{array}{l}
r\in M^\T \setminus F_p\\
\overrightarrow{r\,q}\in E
\end{array}
}
(1-e^{2\pi \il \,w(\overrightarrow{r\,q})})).
\end{equation}
In particular we have that $\mathsf{e}_\T\big( N(F_p)|_p\big)=\ee(p)$.

\begin{lemma}\label{def kirwanclasses}
Let  $(M^{2n},\omega,\T, \psi)$ be a symplectic toric manifold, and let $\mu=\psi^{\overline{\xi}}\colon M\to \R$ be a generic component of the moment map. 
For each $p\in M^\T$ define $\eta_p\in \Ke(M^\T)$ to be
\begin{displaymath} \eta_p(q)=
\begin{cases} 0\ & for\ q \in M^{\T}\setminus F_p\\
\mathsf{e}_\T\big( N(F_p)|_q\big)\ & for\ q \in  F_p\cap M^\T.
\end{cases}
\end{displaymath}
Then $\eta_p$ is an element of $\Ke(M)$, and it is a Kirwan class in the sense of Proposition \ref{existence kirwan}.
\end{lemma}
The element $\eta_p\in \K$ defined in the above lemma is called the {\bf equivariant (K-theoretical) Poincar\'e dual} to the flow-up manifold $F_p$.
\begin{proof}
The proof of this Lemma is quite standard, but we include it here for completeness.
In order to prove that $\eta_p$ is indeed an element of $\Ke(M)$, we need to verify condition \eqref{k cdt} in Theorem \ref{GKM k}, for every edge $e$
of the GKM graph $\Gamma$. 

Let $\overrightarrow{r\,s}\in E$. 
If neither $r$ nor $s$ belong to $F_p$, then by definition of $\eta_p$ we have $\eta_p(r)=\eta_p(s)=0$, so \eqref{k cdt} holds.
If $r\in F_p$ but $s\notin F_p$, then by definition of $\eta_p(r)$ and \eqref{eeKc normal}, $\eta_p(r)=Q\,(1-e^{2\pi \il \,w(\overrightarrow{s\,r})}))$, for some $Q\in R(\T)$,
and \eqref{k cdt} holds. Similarly if  $s\in F_p$ but $r\notin F_p$.
Finally, suppose that both $r$ and $s$ belong to $F_p$. Consider the subtorus $\T'=\exp\{\xi \in Lie(\T)\mid w(\overrightarrow{r\,s})(\xi)=0\}\subset \T$ fixing the sphere
$S^2$ associated to the edge $\overrightarrow{r\,s}$. Since $\T'$ acts trivially on $S^2$ the representations of
$\T'$ on $N(S^2)|_r$ and $N(S^2)|_s$ agree. This implies that there exists an isomorphism $\varphi\colon W_r\to W_s$ such that
for every $w\in W_r$
\begin{equation}\label{cdt isom}\varphi(w)-w=n_w\,w(\overrightarrow{r\,s})\end{equation} for some $n_w\in \Z$. 
Note that $S^2$ is also an invariant submanifold in $F_p$. Denote by $W_{r}^{F_p}$ (resp. $W_{s}^{F_p}$) the set of weights
of the $\T$-representation on the tangent space of $F_p$ at $r$ (resp. at $s$), and observe that the isomorphism $\varphi$ restricts to an isomorphism
from $W_r^{F_p}$ to $W_s^{F_p}$ satisfying \eqref{cdt isom}, and hence to an isomorphism from $W_r\setminus W_r^{F_p}$ to $W_s\setminus W_s^{F_p}$ satisfying \eqref{cdt isom}.  
Observe that 
$e^{2\pi \il \,w}-e^{2\pi \il \,\varphi(w)}=e^{2\pi \il \,w}(1-e^{2\pi \il \,n_w\,w(\overrightarrow{r\,s})})=Q'\,(1-e^{2\pi \il \,w(\overrightarrow{r\,s})}))$,
for some $Q'\in R(\T)$, and that the weights of the $\T$-representation on $N(F_p)|_{r}$ (resp. $N(F_p)|_{s}$) are precisely those in $W_r\setminus W_r^{F_p}$ (resp. $W_s\setminus W_s^{F_p}$). It follows that
$$\mathsf{e}_\T\big( N(F_p)|_r\big)-\mathsf{e}_\T\big( N(F_p)|_s\big)=Q''(1-e^{2\pi \il \,w(\overrightarrow{r\,s})}))$$ for some $Q''\in R(\T)$. 
By Theorem \ref{GKM k} we can conclude that $\eta_p$ is an element of $\Ke(M)$.

In order to finish the proof we need to show that $\eta_p$ satisfies properties (i) and (ii) of Proposition \ref{existence kirwan}. The first property follows from observing that,
by definition of $F_p$, the negative normal bundle
of $\mu$ at $p$ coincides with the normal bundle $N(F_p)$ at $p$. The second one follows from observing that $p$ is a minimum of $\mu$ on $F_p$ and that $F_p$ is connected. Therefore 
every fixed point $q$ in $F_p\setminus \{p\}$ satisfies $\mu(q)>\mu(p)$. Hence any $q'$ with $\mu(q')< \mu(p)$ must be in $M^\T \setminus F_p$, and thus 
$\eta_p(q')=0$.
This concludes the proof.
\end{proof}
\begin{rmk}\label{index eta 1}
In the above proof we show that the normal bundle of $F_p$ at $p$ coincides with the negative normal bundle of $\mu$ at $p$, so $\mathsf{e}_\T(N(F_p)|_p)=\ee(p)$. Therefore
Propostion \ref{prop indexproperties} (ii) implies that
$$
\ii_p(\eta_p)=1\,.
$$

\end{rmk}
\subsection{Toric manifolds: the index increasing case}\label{toric increasing}
In this Subsection we analyze index increasing GKM spaces and prove Proposition \ref{existence canonical} for this case. We use some
standard facts about such spaces which, for the sake of completeness, are proved in Section \ref{ec}.
\begin{prop}\label{classes from cohomology are canonical}
Let  $(M^{2n},\omega,\T, \psi)$ be a symplectic toric manifold, and let $\mu=\psi^{\overline{\xi}}\colon M\to \R$ be a generic component of the moment map.
Let $\Gamma^o=(V,E^o)$ be the associated oriented GKM graph, and assume it is index increasing. 
Then for each $p\in M^\T$, the Kirwan classes $\eta_p$ defined in Lemma \ref{def kirwanclasses} are the i-canonical classes $\tau_p$.
\end{prop}

This proves Proposition \ref{existence canonical} in the index increasing case.
\begin{proof}
We need to show that 
\begin{displaymath} \ii_q (\eta_p)=
\begin{cases} 0\ & for\ q \in M^{\T}\setminus F_p\\
1\ & for\ q \in  F_p\cap M^\T.
\end{cases}
\end{displaymath}
Consider any two fixed points $p,q$. Suppose first that $q \notin F_p$. Then $\eta_p(q)=0$, and therefore the local index at $q$ is $0$ (see Propostion \ref{prop indexproperties}).
Now suppose that $q \in F_p$. By Lemma \ref{lemma Fp bigger index} we have $\lambda_q > \lambda_p$.
In the GKM graph of $M$ the vertex $q$ is connected to $\lambda_q$ vertices $q_1,\ldots,q_{\lambda_q}$ by edges terminating at $q$, and to $n-\lambda_q$ vertices, $q_{\lambda_q+1},\ldots,q_n$ by edges starting at $q$.  As the moment map is index increasing, 
from Corollary \ref{edges in} it follows that
the vertices $q_{\lambda_q+1},\ldots,q_n$ are also in $F_p$, as $q$ is.
Therefore $\lambda_p$ of the points $q_1,\ldots,q_{\lambda_q}$ are not in $F_p$ (see also Corollary \ref{edges in}). To simplify the notation assume that $q_1,\ldots,q_{\lambda_p}$ are not in $F_p$. The value of $\eta_p$ at these points is $0$. Let $w_1,\ldots,w_{\lambda_q}$ be the weights of $\T$-action on the tangent spaces at $q$ of the spheres corresponding to the edges $\overrightarrow{q_1\,q},\ldots, \overrightarrow{q_{\lambda_q}\,q}$. Then by \eqref{eeKc normal}  the value of $\eta_p$ at $q$ is given by
$$\eta_p(q)=\prod_{j=1}^{\lambda_p}\,(1-\e{w_j}).$$
To calculate the local index of $\eta_p$ at $q$ we look at the class $\kappa(\widetilde{r}(\eta_p))$ in $\mathcal{K}_{\widetilde{\T}_{q}}(\widetilde{H}^{\epsilon}_q)$. Using the algorithm and notation from Section \ref{section local index} we find the values of this class at the fixed points of $\widetilde{H}^{\epsilon}_q$:
\begin{align*}
\kappa(\widetilde{r}(\eta_p)) (p_l)&=\prod_{j=1}^{\lambda_p}\,(1-\e{(w_j-w_l)})=0,&l=1,\ldots,\lambda_p,\\
\kappa(\widetilde{r}(\eta_p)) (p_l)&=\prod_{j=1}^{\lambda_p}\,(1-\e{(w_j-w_l)})\neq 0,&l=\lambda_p+1,\ldots,\lambda_q,\\
\kappa(\widetilde{r}(\eta_p)) (p_0)&=\prod_{j=1}^{\lambda_p}\,(1-\e{(w_j+w_0)})\neq 0. &
\end{align*}

Note that at the points $p_l$ where the value of $\kappa(\widetilde{r}(\eta_p))$ is nonzero, it is exactly equal to the product of terms $(1-\e{w})$ taken over the weights $w$ on the edges connecting $p_s$ to fixed points in $(\widetilde{H}^{\epsilon}_q)\cong \C \pp^{\lambda_q}$ where the value of $\kappa(\widetilde{r}(\eta_p))$ is $0$. This observation, and Atiyah-Segal formula \eqref{AS index}, imply that the index of $\kappa(\widetilde{r}(\eta_p))$ is equal to the index of a class $\mathbf{1}$ on $\C \pp^{\lambda_q-\lambda_p}\subset \C \pp^{\lambda_q}$ containing the fixed points of $ \C \pp^{\lambda_q}$ where $\kappa(\widetilde{r}(\eta_p))$ is nonzero. By \eqref{index 1-2} we have that $ \textrm{Ind}(\mathbf{1}_{\C \pp^{\lambda_q-\lambda_p}})= 1$.
 In other words, $\textrm{Ind}( \kappa (\eta_p))$ is equal to
\begin{displaymath}
 \begin{aligned}
&\frac{\displaystyle\prod_{j=1}^{{\lambda_p}}(1-\e{(w_j+w_0)})}
{\displaystyle\prod_{j=1}^{{\lambda_q}}(1-\e{(w_j+w_0)})}
+0+\sum_{l={\lambda_p}+1}^{\lambda_q} \,\frac{\displaystyle\prod_{j=1}^{{\lambda_p}}(1-\e{(w_j-w_l)})}{(1-\e{(-w_l-w_0)})\displaystyle\prod_{j=1,j\neq l}^{{\lambda_q}}(1-\e{(w_j-w_l)})}=\\
& \frac{1}{\displaystyle\prod_{j={\lambda_p}+1}^{{\lambda_q}}(1-\e{(w_j+w_0)})}+\sum_{s={\lambda_p}+1}^{\lambda_q} \,\frac{1}{(1-\e{(-w_l-w_0)})
\displaystyle\prod_{j={\lambda_p}+1,j\neq l}^{{\lambda_q}}(1-\e{(w_j-w_l)})}= \\
& \textrm{Ind}(\mathbf{1}_{\C \pp^{\lambda_q-\lambda_p}})= 1.
\end{aligned}
\end{displaymath}
This shows that the classes defined naturally as the equivariant Poincar\'e duals to the flow-up manifolds $F_p$ are i-canonical classes.
\end{proof}

\begin{exm}[The complex projective space $\mathbb{C}\mathbb{P}^n$.]\label{projective space}

Consider the complex projective space $(\mathbb{C}\mathbb{P}^n,\omega)$ where $\omega$ denotes the Fubini-Study symplectic form rescaled so that $[\omega]$ is integral and primitive, i.e.\ $[\omega]$ is a generator of $H^2(\mathbb{C}\mathbb{P}^n;\Z)=\Z$, regarded as a sublattice of
$H^2(\mathbb{C}\mathbb{P}^n;\R)$. Endow $(\mathbb{C}\mathbb{P}^n,\omega)$ with the standard toric action of an $n$-dimensional torus $\T$ and moment map $\psi$; as before, let $\mathfrak{t}$ be the Lie algebra of $\T$ and $\ell\subset \mathfrak{t}$ the integral lattice. Since the action is Hamiltonian, the symplectic form extends to an equivariant form in the Cartan complex, called equivariant symplectic form, given by
$\omega+\psi$. We can choose the moment map so that $\psi(q)\in \ell^*$ for every
$q\in (\mathbb{C}\mathbb{P}^n)^\T$. Then the above equivariant form
represents a well-defined element $[\omega+\psi]$ in $H^2_{\T}(\C\pp^n;\Z)$ (regarded as a sublattice of $H^2_\T(\C\pp^n;\R)$). At the level of the GKM graph $(V,E)$ associated
to $(\mathbb{C}\mathbb{P}^n,\omega,\T,\psi)$, the condition of $[\omega]$ being integral and primitive translates into saying that
\begin{equation}\label{all 1}
\psi(q)-\psi(p)=w(\overrightarrow{p\,q})\quad\mbox{for every}\quad \overrightarrow{p\,q}\in E.
\end{equation}
Indeed, this is equivalent to saying that the symplectic volume of the sphere $\psi^{-1}(\overrightarrow{p\,q})$ is $1$, for every $\overrightarrow{p\,q}\in E$. Note that
for every pair of fixed points $p,q\in (\C\pp^n)^\T$ there exists an edge $\overrightarrow{p\,q}\in E$.

Consider the equivariant K-theory class represented by the equivariant line bundle $\mathbb{L}^{S^1}$ satisfying $\mathrm{c}_{1}^{S^1}(\mathbb{L}^{S^1})=-[\omega+\psi]$, where $\mathrm{c}_1^{S^1}(\mathbb{L}^{S^1})$ denotes the equivariant first Chern class of $\mathbb{L}^{S^1}$.
Such bundle exists by Theorem 1.1 and Corollary 1.2 proved by Hattori and Yoshida in \cite{HY}. 
Note that $\mathbb{L}^{S^1}(s)=e^{-2\pi \mathrm{i}\psi(s)}$ for every $s\in (\C\pp^n)^\T$.
Let $\overline{\xi}\in \mathfrak{t}$ be a generic vector, and consider the ordering induced by $\mu=\psi^{\overline{\xi}}$ on the fixed points.
Note that for every choice of generic vector $\overline{\xi}$, the oriented GKM graph associated to $(\mathbb{C}\mathbb{P}^n,\omega,\T,\psi,\overline{\xi})$ is index increasing.

For every $p\in M^\T$ consider the class
$$
\tau_p=\prod_{\stackrel{q\in M^\T}{q\prec p}} \big(1-e^{2\pi \mathrm{i}\psi(q)}\mathbb{L}^{S^1}\big).
$$
Observe that $q\prec p$ if and only if $q \notin F_p$ (see also Proposition \ref{proposition fp is vpplus}). Therefore, for each $s\in (\C\pp^n)^\T$, we have that 
$$\tau_p(s)=
\begin{cases}
\displaystyle 0 & \mbox{if} \quad s\notin F_p \\
\displaystyle\prod_{\stackrel{q\in M^\T}{q\prec p}} \big(1-e^{2\pi \mathrm{i}(\psi(q)-\psi(s))}\big)=
\prod_{\stackrel{q\in M^\T}{q\notin F_p}} \big(1-e^{2\pi \mathrm{i}w(\overrightarrow{qs})}\big) & \mbox{if}\quad s\in F_p.
\end{cases}
$$
Thus $\tau_p$ coincides with the equivariant (K-theoretical) Poincar\'e dual to $F_p$ which, by 
Proposition \ref{classes from cohomology are canonical}, is the i-canonical class at $p$.
\end{exm}
The i-canonical classes for $\C\pp^2$ are shown in
Figure \ref{fig:cp2classes}. 
\begin{figure}[htbp]
\begin{center}
\psfrag{0}{\color{blue}$0$}
\psfrag{1}{\color{blue}$1$}
\psfrag{t0}{$\tau_0$}
\psfrag{t1}{$\tau_1$}
\psfrag{t2}{$\tau_2$}
\psfrag{ex}{\color{blue}{$1-e^{2 \pi i x}$}}
\psfrag{ey}{\color{blue}{$1-e^{2 \pi i y}$}}
\psfrag{exy}{\color{blue}{$(1-e^{2 \pi i (y-x)})(1-e^{2 \pi i y})$}}
\psfrag{x}{\tiny{$x$}}
\psfrag{y}{\tiny{$y$}}
\psfrag{y-x}{\tiny{$y-x$}}
\includegraphics[width=10cm]{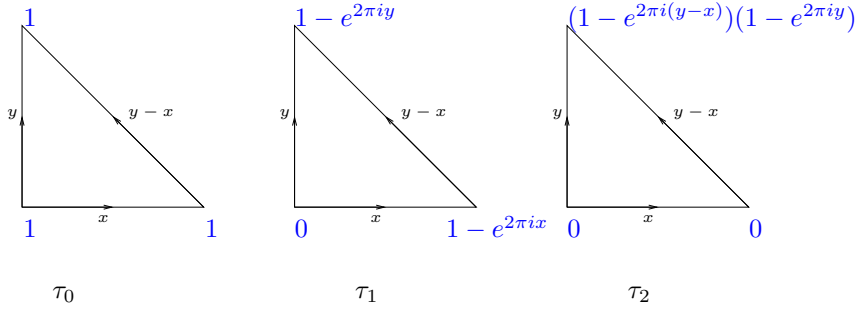}
\caption{The basis of i-canonical classes $\{\tau_p\}$ for $\C \pp^2$.}
\label{fig:cp2classes}
\end{center}
\end{figure}

\subsection{Toric manifolds: the non-index increasing case}\label{toric non-increasing}
 
In the case where the moment map is not index increasing we can still associate to each fixed point $p \in M^\T$ a K-theory class $\eta_p$, as in Lemma \ref{def kirwanclasses}. The value of $\eta_p$ on points $q \in M^\T \setminus F_p$ is zero and therefore the local index $\ii_q(\eta_p)$ is also zero (see Propostion \ref{prop indexproperties}). However it may no longer be true that $\ii_q(\eta_p)=1$ for $q \in F_p\cap M^\T=: F_p^\T$. 
Below we present an algorithm to construct i-canonical classes $\tau_p$ out of the equivariant Poincar\'e duals $\eta_p$ as in Definition \ref{def kirwanclasses}. 
\begin{defin}\label{def vpplus}
For any $p \in M^\T$ we define $V_p^+$ to be the set of fixed points which can be joined to $p$ through an increasing path in $E^o$, i.e. $q\in V_p^+$ if and only if there exists
a sequence of edges $\gamma=(r_0=p,r_1,\ldots,r_m=q)$ such that $\overrightarrow{r_{i}r_{i+1}}\in E^o$ for every $i=0,\ldots,m-1$. 
\end{defin}
This definition implies that $\mu (q) > \mu (p)$, i.e. $p \prec q$, for all $q \in V_p^+ \setminus \{p\}$.
Also observe that (by Lemma \ref{one inclusion}) $F_p^\T \subset V_p^+$ and thus for all $q \in F_p^\T$ we have that $F_q^\T \subseteq V_q^+ \subset V_p^+$. (In the index increasing case a stronger statement is true: $F_q \subseteq F_p$ for all $q \in F_p^\T$, see Corollary \ref{edges in}). 
\begin{proof}[Proof of Proposition \ref{existence canonical} - the non-index increasing case]
Fix $p \in M^\T$ and let  $V_p^+= \{q_0=p,q_1,\ldots,q_k\}$ be ordered so that  $q_j \prec q_l$ for $0\leq j<l\leq k$. 
As the restriction of $\mu$ to $F_{q_l}$ attains its minimum at $q_l$, we have that $$j<l \Rightarrow q_j \notin F_{q_l}  \textrm{  and therefore  } \eta_{q_l}(q_j)=0.$$ 
We inductively construct auxiliary classes $a_1,\ldots,a_k$ satisfying
\begin{itemize}
\item $\ii_{q_j}(a_l)=1$ if $j\leq l$ and $ q_j \in F_p^\T$; \\
\item $\ii_{q_j}(a_l)=0$ if $j\leq l$ and $ q_j \in V_p^+\setminus F_p^\T$;\\
\item $\ii_{q}(a_l)=0$ if $ q \notin V_p^+$. 
\end{itemize}
and define the i-canonical class $\tau_p$ to be $a_k$. 
In the following we will make use of the fact that
$F_p^\T\subset V_p^+$.

Define
\begin{displaymath}
a_1:= \begin{cases} \eta_p + (1-\ii_{q_1}(\eta_p))\ \eta_{q_1} & \textrm{ if }q_1 \in F_p^\T,\\
\eta_p -\ii_{q_1}(\eta_p)\ \eta_{q_1} & \textrm{ if }q_1 \in V_p^+ \setminus F_p^\T.\\
                 \end{cases}
\end{displaymath}
Then, as $\eta_{q_1}(p)=0$, we have $a_1(p)= \eta_p(p)=\ee(p)$, and hence  $\ii_p(a_1)=\ii_p(\eta_p)=1$ (see Propostion \ref{prop indexproperties} and Remark \ref{index eta 1}). 
Also, observe that if $\eta_{q_1}(s)\neq 0$ for some $s\in M^\T$ then $s\in F_{q_1}^\T \subset V_{q_1}^+\subset V_p^+$, where the first inclusion follows from Lemma \ref{one inclusion}, and the second is obvious.
Thus the class $a_1$ restricts to zero on $M^\T \setminus V_p^+$ and $\ii_{s}(a_1)=0$ if $s\notin V_p^+$.
Moreover, by Proposition \ref{prop indexproperties} and Remark \ref{index eta 1} we can conclude that
\begin{displaymath}
\ii_{q_1}(a_1):= \begin{cases} \ii_{q_1}(\eta_p)+(1-\ii_{q_1}(\eta_p))\cdot 1=1 & \textrm{ if }q_1 \in F_p^\T,\\
\ii_{q_1}(\eta_p)-\ii_{q_1}(\eta_p)\cdot 1=0 & \textrm{ if }q_1 \in V_p^+ \setminus F_p^\T.
                 \end{cases}
\end{displaymath}
Then we proceed inductively and define
\begin{displaymath}
 a_j= \begin{cases} a_{j-1}+ (1- \ii_{q_j}(a_{j-1}))\eta_{q_{j}} & \textrm{ if }q_j \in F_p^\T,\\
       a_{j-1}- \ii_{q_j}(a_{j-1})\eta_{q_j}& \textrm{ if }q_j \in V_p^+ \setminus F_p^\T.
      \end{cases}
\end{displaymath}
As the fixed points are ordered with $\prec$, the restrictions of $\eta_{q_j}$ to fixed points $q_0,q_1,\ldots, q_{j-1}$ are zero. Thus the local index of $a_j$ at $q_l$, $l=0,\ldots,j-1$ is the same as of $a_{j-1}$. Similarly as before, 
 Remark \ref{index eta 1} and Proposition \ref{prop indexproperties} prove that $\ii_{q_j}(a_j) =1$ if $q_j \in F_p^\T$ and is zero on $V_p^+\setminus F_p^\T$. 
Moreover $a_j$ restricts to zero on $M^\T \setminus V_p^+$ and $\ii_{s}(a_j)=0$ if $s\notin V_p^+$.

The algorithm ends when we exhaust all the points in $V_p^+= \{q_0,q_1,\ldots,q_k\}$ and we define the class $\tau_p$ to be $a_k$. Thus $\ii_q(\tau_p)=1$ if $q \in F_p^\T$ and  $\ii_q(\tau_p)=0$ if $q \notin F_p^\T$ as needed. What is left to prove is that the classes $\tau_p$ are Kirwan classes in the sense of Proposition \ref{existence kirwan}. This follows immediately from observing that 
$\tau_p=\eta_p+\displaystyle\sum_{q_l\in V_p^+\setminus\{p\}}\alpha_l\,\eta_{q_l}$, where $\alpha_l\in R(\T)$ for every $l$, and from the $\eta_{q_l}$'s also being Kirwan classes.
\end{proof}
\begin{rmk} \label{rmk basis}
Note that the collection of i-canonical classes $\{\tau_p\}_{p\in M^\T}$ is obtained from the basis of Kirwan classes $\{\eta_p\}_{p\in M^\T}$ by applying a lower triangular matrix with $1$'s on diagonal.
This gives an alternative proof that they form a basis of $\K$.
\end{rmk}
Figure \ref{fig:hirzebruchclasses} presents the basis of i-canonical classes for the Hirzebruch surface. 
\begin{figure}[htbp]
\begin{center}
\psfrag{0}{\color{blue}{\footnotesize{$0$}}}
\psfrag{1}{\color{blue}\footnotesize{$1$}}
\psfrag{t0}{\footnotesize{$\tau_0$}}
\psfrag{t1}{\footnotesize{$\tau_1$}}
\psfrag{t2}{\footnotesize{$\tau_2$}}
\psfrag{t3}{\footnotesize{$\tau_3$}}
\psfrag{expy}{\color{blue}\footnotesize{$1-e^{2 \pi i (x+y)}$}}
\psfrag{exmy}{\color{blue}\footnotesize{$(1-e^{2 \pi i (x-y)})e^{2 \pi i y}$}}
\psfrag{ey}{\color{blue}\footnotesize{$1-e^{2 \pi i y}$}}
\psfrag{eyyx}{\color{blue}\footnotesize{$(1-e^{2 \pi i y})e^{2 \pi i (y-x)}$}}
\psfrag{eyymx}{\color{blue}\footnotesize{$(1-e^{2 \pi i y})(1-e^{2 \pi i (y-x)})$}}
\psfrag{y}{\tiny{$y$}}
\psfrag{y-x}{\tiny{$y-x$}}
\psfrag{x+y}{\tiny{$x+y$}}
\includegraphics[width=14cm]{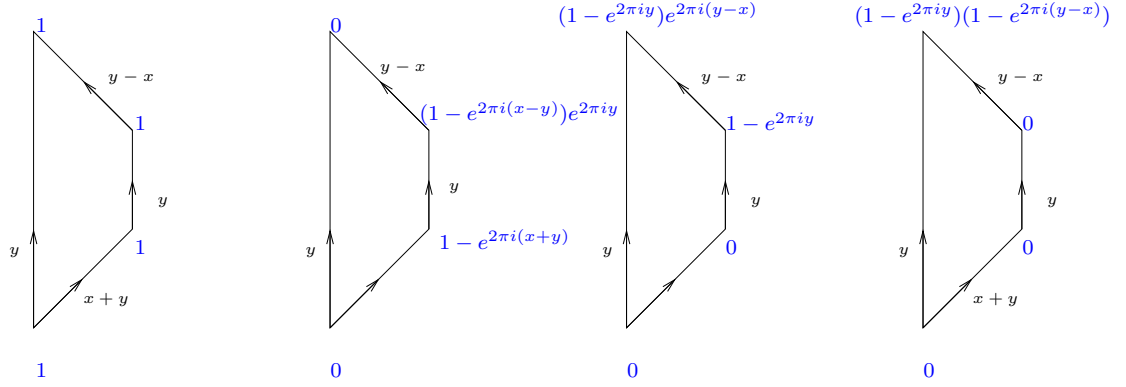}
\caption{The basis of i-canonical classes $\{\tau_p\}$ for the Hirzebruch surface.}
\label{fig:hirzebruchclasses}
\end{center}
\end{figure}

\section{Applications to equivariant cohomology}\label{ec}
As it was already observed by Guillemin and Kogan, it is possible to use the local index map to define ``canonical classes"
forming a basis of the equivariant cohomology ring for every Hamiltonian $\T$-space, in particular for symplectic toric manifolds (see \cite[Remark 1.4]{GK}). We follow their idea, but use a different definition of local index. Moreover we relate our basis to the bases already commonly used in equivariant cohomology. 
 Namely the basis thus obtained coincides with the basis of the equivariant Poincar\'e duals $\{\eta_p\}$ defined below. 
In the addition, when the generic component of the moment map is the index increasing, Goldin and Tolman \cite{GT} introuduced another basis for the equivariant cohomology (see Definition \ref{definition gtclasses}).
In this case we show that the three sets of bases - the i-canonical classes, the equivariant Poincar\'e duals $\{\eta_p\}$, and the Goldin-Tolman \cite{GT} canonical classes -  are the same.

Our definition of i-canonical classes for equivariant cohomology is slightly different from the definition given in K-theory. Namely, here we require the local index of a class associated to a fixed point $p$ to vanish on all $M^\T \setminus \{p\}$. The reason for this change is that we would like the class $\mathbf{1}_M$ to be an element of our i-canonical basis.

Recall that in the Borel description, the equivariant cohomology ring $H_\T^*(M;\Z)$ is defined to be the ordinary cohomology ring of the space $M\times_{\T} E\T$, where $E\T$ is a contractible space on which $\T$ acts freely. If $\T$ is a $d$-dimensional torus, and $x_1,\ldots,x_d$ a $\mathbb{Z}$-basis of the dual lattice of $\T$, then
$H_\T^*(\p;\Z)=H^*((\C\mathbb{P}^{\infty})^d;\Z)=\Z[x_1,\ldots,x_d]$. The unique map $M\to \p$ gives $H_\T^*(M;\Z)$ the structure of an $H_\T^*(\p;\Z)$-module.

Let $(M,\omega,\T,\psi)$ be a GKM space (see Section \ref{GKM section}) with GKM graph $(V,E)$. For an edge $\overrightarrow{p\,q}\in E$,
let $\T'=\T'_{\overrightarrow{p\,q}}=\exp\{\xi \in \mathfrak{t}\mid w(\overrightarrow{p\,q})(\xi)=0\}$ be the subtorus fixing the sphere $S^2_{\overrightarrow{p\,q}}=\psi^{-1}(\overrightarrow{p\,q})$ corresponding to the edge $\overrightarrow{p\,q}$. 
Let $\mathbb{S}(\mathfrak{t}^*)$ be the symmetric algebra of $\mathfrak{t}^*$ and $\pi_{\overrightarrow{p\,q}}\colon \mathbb{S}(\mathfrak{t}^*)\to \mathbb{S}((\mathfrak{t}')^*)$ be the projection induced by the inclusion of $\mathfrak{t}'=Lie(\T')$ into $\mathfrak{t}$.
Then for any class $\tau\in H_\T^*(M;\Z)$ we must have 
\begin{equation}\label{GKM cd eq}
\pi_{\overrightarrow{p\,q}}(\tau(q))=\pi_{\overrightarrow{p\,q}}(\tau(p))\quad\mbox{for every}\quad \overrightarrow{p\,q}\in E.
\end{equation}
This condition is necessary, but not sufficient to guarantee that a class $\tau\in H_\T^*(M^\T;\Z)$ is in $H_\T^*(M;\Z)$. However, if working with rational coefficients, 
a theorem of Goresky-Kottwitz-MacPherson (\cite{GKM}) implies that
each $\tau$ in $H_\T^*(M^\T;\Q)$ satisfying \eqref{GKM cd eq} is in $H_\T^*(M;\Q)$. (Compare with Theorem \ref{GKM k}.)

\subsection{Choosing a basis.}
We start by recalling the choices of bases already used in the literature: the basis consisting of equivariant Poincar\'e duals to flow up manifolds and the Goldin-Tolman canonical classes in the index increasing case.

Let $(M,\omega,\T,\psi)$ be a Hamiltonian $\T$-space, and let $\mu\colon M\to\R$ be a generic component of the moment map.
For a fixed point $p$, the {\bf equivariant (cohomological) Euler class} of the negative normal bundle $N_p^-$ of $\mu$ at $p$ is the element 
$\Lambda_p^-\in H^{2\lambda_p}(\{p\};\Z)$ given by 
\begin{equation}\label{ennb}
\Lambda_p^-=\prod_{w_j\in W_p^+}w_j\,.
\end{equation}
The following Proposition is due to Kirwan. 
\begin{prop}[Kirwan]\label{existence kirwan cohom}
Let $(M,\omega,\T,\psi)$ be a Hamiltonian $\T$-space, and let $\mu\colon M\to\R$ be a generic component of the moment map. 
Then for every $p\in M^\T$ there exists a class $\nu_p\in H_\T^{2\lambda_p}(M;\Z)$ such that 
\begin{itemize}
\item[(i)] $\nu_p(p)=\Lambda_p^-$;
\item[(ii)] $\nu_p(q)=0$ for every $q\in M^\T$ such that $q \prec p$ (i.e. $\mu(q)< \mu(p)$).
\end{itemize}
Moreover the set $\{\nu_p\}_{p\in M^\T}$ is a basis for $H_\T^*(M;\Z)$ as a module over $H_\T^*(\p;\Z)$.
\end{prop}
An equivariant cohomology class satisfying properties (i) and (ii) above is called a {\bf Kirwan class} (at the fixed point $p$). 

Note that Proposition \ref{existence kirwan} is a generalization of the above original result of Kirwan from the equivariant cohomology setting to the K-theory setting. Due to this similarity we omit the proof of Proposition \ref{existence kirwan cohom}, which is based on the fact that $\Lambda_p^-$ is not a zero divisor in $H_\T(\p;\Z)$.

Henceforth $(M,\omega,\T,\psi,\overline{\xi})$ denotes a symplectic toric manifold endowed with a choice of generic $\overline{\xi}\in \mathfrak{t}$.
In analogy with what we did in Section \ref{section construction} for equivariant K-theory, we
define the equivariant (cohomological) Poincar\'e duals to the flow-up manifolds $F_p$.  
Let $F_p$ be the flow-up manifold at the fixed point $p$, and $N(F_p)$ the normal bundle of $F_p$ in $M$. Then
 the equivariant (cohomological) Euler class of the normal bundle $N(F_p)$ of $F_p$ in $M$ at a fixed point $q$ is
an element in $H^{2\lambda_p}(\{q\};\Z)$  given by
\begin{equation}\label{eec normal}
\chi_\T\big( N(F_p)|_q\big)=\prod_{
\begin{array}{l}
r\in M^\T \setminus F_p\\
\overrightarrow{r\,q}\in E
\end{array}
}
w(\overrightarrow{r\,q}).
\end{equation}
In particular, since $N(F_p)|_p=N_p^-$, we have that $\chi_\T\big( N(F_p)|_p\big)=\Lambda_p^-$.

\begin{defin}\label{eqcoh poincare duals}
Let  $(M,\omega,\T, \psi)$ be a symplectic toric manifold, and let $\mu=\psi^{\overline{\xi}}\colon M\to \R$ be a generic component of the moment map. 
For each $p\in M^\T$ define the {\bf equivariant Poincar\'e dual} to the flow-up manifold $F_p$ to be the class
$\eta_p\in H_\T^{2\lambda_p}(M;\Z)$ whose restriction to the fixed points is given by
\begin{displaymath} \eta_p(q)=
\begin{cases} 0\ & for\ q \in M^{\T}\setminus F_p\\
\chi_\T\big( N(F_p)|_q\big)\ & for\ q \in  F_p\cap M^\T.
\end{cases}
\end{displaymath}
\end{defin}
It is easy to check that $\eta_p$ is a Kirwan class in the sense of Proposition \ref{existence kirwan cohom} for every $p\in M^\T$, and thus $\{\eta_p\}_{p \in M^\T}$ forms a basis of $H_\T^*(M;\Z)$ as a module over $H^*_\T(\p;\Z)$.

Goldin and Tolman \cite{GT} define 
another basis for the equivariant cohomology ring of a symplectic manifold endowed with a Hamiltonian torus action. 
\begin{defin} \label{definition gtclasses}
 Let $(M,\omega,\T,\psi)$ be a Hamiltonian $\T$-space, and let $\mu\colon M\to\R$ be a generic component of the moment map. 
 A cohomology class $\zeta_p \in H_\T^{2\lambda_p}(M; \Z)$ is a canonical class in the sense of Goldin and Tolman (henceforth referred to as {\bf GT-canonical class}) at $p\in M^\T$  if
 \begin{itemize}
\item[(i)] $\zeta_p(p)=\Lambda_p^-$;
\item[(ii)] $\zeta_p(q) =0$ for all $q \in M^\T \setminus \{p\}$ such that $\lambda_q \leq \lambda_p$.
\end{itemize}
\end{defin}
GT-canonical classes do not always exist; however, if they exist, they are uniquely associated to the chosen component of the moment map $\mu$ (see \cite[Lemma 2.7]{GT}). 
Moreover GT-canonical classes are Kirwan classes in the sense of Proposition \ref{existence kirwan cohom} (cf. \cite[Lemma 2.8]{GT}), hence if they exist for every fixed point $p$, they form a basis of $H^*_\T(M;\Z)$ as a module over $H_\T^*(\p;\Z)$.

If the $\T$ action above is GKM, and if the chosen component of the moment map $\mu$ is index increasing, then for each fixed point $p$ the GT-canonical class $\zeta_p$ exists. Conversely, if for all $p \in M^\T$ the GT-canonical class exist, then $\mu$ must be index increasing (see Theorem 4.1 and Remark 4.2 in \cite{GT}).

Below we propose a different choice of basis for the equivariant cohomology ring, making use of our definition of local index translated to the equivariant cohomology case. 

Similarly to what we did in Section \ref{section local index}, for a fixed point $q \in M^\T$ we define the map $\widetilde{r} \colon H^*_\T(M;\Z) \rightarrow H_{\T\times S^1}^*(H_q \times S^2;\Z)$ and denote by $\kappa$ the Kirwan map $\kappa \colon H_{\T\times S^1}^*(H_q \times S^2;\Z) \rightarrow H^*_{\widetilde{\T}_{q}}(\widetilde{H}^{\epsilon}_q;\Z)$ which is surjective (for surjectivity of $\kappa$ over the integers see \cite[Proposition 7.3]{TWquotients}).  
The index homomorphism in equivariant cohomlogy, corresponding to the equivariant K-theory push-forward map, is simply given by integration. 
By the Atiyah-Bott-Berline-Vergne Localization Formula \cite{AB,BV} this integral 
$$\ii \colon H^*_{\widetilde{\T}_{q}}(\widetilde{H}^{\epsilon}_q;\Z) \rightarrow H^{*-2\lambda_q}_{\widetilde{\T}_{q}}(\p ;\Z)$$
is the following sum of rational functions:
$$
 \ii (\alpha)= \int_{\widetilde{H}^{\epsilon}_q} \alpha=\sum_{q_j \in (\widetilde{H}^{\epsilon}_q)^{\T_q}}\frac{\alpha(p)}{\displaystyle \prod_{w_l \in W_{q_j}}w_l}.$$

Observe that if $\alpha \in H^{2\lambda}_{\widetilde{\T}_{q}}(\widetilde{H}^{\epsilon}_q;\Z)$ and $\lambda<\lambda_q$ then $\ii(\alpha)=0$.

For a class $\tau \in H_\T^*(M;\Z)$ let the
{\bf local index of $\tau$ at $q$}, denoted by $\ii_q(\tau)$, be the image of the above $\ii (\kappa \circ \widetilde{r} (\tau)) \in H^{*-2\lambda_q}_{\widetilde{\T}_{q}}(\p;\Z)$  under the natural homomorphism $H^*_{\widetilde{\T}_{q}}(\p;\Z) \rightarrow H^*_\T(\p;\Z)$ (compare with Definition \ref{definition local index}).
\begin{rmk}\label{index prop coh}
An analogous of Proposition \ref{prop indexproperties} holds for the local index in equivariant cohomology (the only difference is that $R(\T)$ from the K-theory setting needs to be replaced by $H^*_\T(\p;\Z)$). 
Moreover, by comparing the degrees, one sees that if $\tau \in  H_\T^{2 \lambda}(M;\Z)$ then $\ii_q(\tau)=0$ for all $q \in M^\T$ with $\lambda<\lambda_q$, and that for all $q \in M^\T$ with $\lambda \geq \lambda_q$ the degree of $\ii_q(\tau) \eta_q$ is equal to the degree of $\tau$ (here $\eta_q$ is the equivariant Poincar\'e dual to $F_q$).
 \end{rmk}
 \begin{defin}\label{definition canonical classes cohom} 
Let $(M,\omega,\T,\psi)$ be a symplectic toric manifold of dimension $2n$, together with 
a choice of a generic component of the moment map $\mu=\psi^{\overline{\xi}} \colon M\to \R$. Then for each $p\in M^\T$,
a Kirwan class $\tau_p \in H_\T^*(M;\Z)$ satisfying
\begin{enumerate}
 \item $\ii_p(\tau_p)=1$,
\item $\ii_q(\tau_p)=0$ for all points $q \in M^\T\setminus \{p\}$,
\end{enumerate}
is called an {\bf i-canonical class} at (the fixed point) $p$.
 \end{defin}
 \begin{rmk}
 Note that as the above classes are Kirwan, they satisfy $\tau_p(p)=\Lambda_p^-$, $\tau_p(q)=0$ for $q \prec p$, and they form a basis of $H_\T^*(M;\Z)$ as an $H_\T^*(\p;\Z)$-module. (We will see later that in fact $\tau_p$ is zero at all the points $q \in M^\T \setminus V_p^+$.)
By repeating the argument of Proposition \ref{uniqueness} one can show that if i-canonical classes exist they are unique.
 \end{rmk}
In what follows we will prove the existence of i-canonical classes, compare them with the equivariant Poincar\'e duals
to the flow-up manifolds, and
with the GT-canonical classes in the index increasing case. 
Before doing so we will prove some technical lemmas.
\subsection{Technical Lemmas}\label{TL}
We recall that given a moment map $\mu\colon X\to \R$ for an $S^1$ action, the subset of $X$ where $\mu$ achieves its minimum is connected.
Hence, if the action has only isolated fixed points, there exists a unique point $p\in X$ where $\mu$ acheives its minimum (see \cite[Lemma 5.1]{GS}). 
\begin{lemma} \label{lemma path to minimum}
Let $(X, \omega,\T,\psi)$ be a symplectic toric manifold together with a choice of a generic component $\mu\colon M\to \R$ of the moment map inducing an orientation on the associated GKM graph $(V,E)$.
Let $p$ be the  vertex corresponding to the fixed point where $\mu$ attains its minimum. Then for every vertex $q \in V$ there exists an increasing path $\overrightarrow{p \,p_k}, \ldots, \overrightarrow{p_1\, q}$ from $p$ to $q$.
\end{lemma}
\begin{proof} 
The unique point $p$ where $\mu$ acheives its minimum is the fixed point with only negative 
weights (i.e.\ it corresponds to the only vertex with no incoming edges in the oriented GKM graph, that is $W_p^+=\emptyset$).
Take any vertex $q \in V$. If $q=p$ we are done. Otherwise there exists a vertex $p_1$ and an 
 edge $\overrightarrow{p_1\, q}\in E$ with $w(\overrightarrow{p_1\, q})\in W_q^+$, hence $\overrightarrow{p_1\, q}\in E^o$. If $p_1=p$ we are done. Otherwise we continue this process and construct a path $\overrightarrow{p_{l}\,p_{l-1}}, \ldots, \overrightarrow{p_1\, q}$ in the oriented GKM graph. As the GKM graph is finite and connected this procedure must end at $p$.
\end{proof}
We recall that given a symplectic toric manifold $(M,\omega,\T,\psi,\overline{\xi})$ with oriented GKM graph $(V,E^o)$ (which is not necessarily index increasing) and given $p\in M^\T$, $V_p^+$ is defined to be the set of vertices which can be joined to $p$ through an increasing path in $E^o$ (see Definition \ref{def vpplus}). The previous lemma, applied to $X=F_p$, implies the following.
\begin{lemma}\label{one inclusion}
Given a symplectic toric manifold $(M,\omega,\T,\psi,\overline{\xi})$, for every $p\in M^\T$ we have that
$$
F_p^\T \subset V_p^+
$$
\end{lemma}
In Proposition \ref{proposition fp is vpplus} we will prove that the opposite inclusion also holds provided that $(M,\omega,\T,\psi,\overline{\xi})$ is index increasing.
\\

{\bf \underline{Convention:}}
In the rest of subsection \ref{TL}, $(M,\omega, \T, \psi, \overline{\xi})$ denotes a compact symplectic toric manifold together a choice of generic $\overline{\xi} \in \mathfrak{t}$ such that the corresponding moment map component $\mu=\psi^{\overline{\xi}}\colon M\to \R$ is \emph{\underline{index increasing.}} The associated oriented GKM graph is denoted by $\Gamma^o=(V,E^o)$. These hypotheses and notation apply to all of the following lemmas, propositions and corollaries.

\begin{lemma}\label{lemma Fp bigger index}
Let $p \in M^\T$ and $(F_p,\omega_{|F_p})$ the flow up at $p$ (see definition in Section \ref{section local index}). Then for any $q \in F_p^\T \setminus \{p\}$ we have that $\lambda_q > \lambda_p$.
\end{lemma}
\begin{proof} 
Let $\T_p^0\subset \T$ be the subtorus acting trivially on $F_p$. We recall that $(F_p,\omega_{|F_p})$ together with the action of $\T / \T_p^0$ is a symplectic toric manifold.
Moreover we can direct the edges of the corresponding GKM graph $\widetilde{\Gamma}$ by using $\mu|_{F_p}$, thus obtaining an oriented GKM graph $\widetilde{\Gamma}^o$ which is a
subgraph of the oriented GKM graph associated to $(M,\omega, \T, \psi, \xi)$. The oriented graph associated to $F_p$ is not necessarily index increasing with respect to indices computed in $F_p$. 
However, by Lemma \ref{lemma path to minimum}, there exists an oriented path in $\widetilde{\Gamma}^o$ from $p$ to $q$ such that $\mu$ increases on each edge. 
Therefore the index (in $M$) also increases on each edge, which proves that $\lambda_q > \lambda_p$.
\end{proof}

\begin{prop}\label{pd is gt}
Let $\eta_p \in H^{2\lambda_p}_{\T}(M;\Z)$ be the equivariant Poincar\'e dual to $F_p$ (see Definition \ref{eqcoh poincare duals}) and $\zeta_p \in H^{2\lambda_p}_{\T}(M;\Z)$ the GT-canonical class at $p\in M^\T$. Then $\eta_p= \zeta_p$.
\end{prop}

\begin{proof}
Notice that our index increasing assumption implies that GT-canonical classes exist at each $p\in M^\T$.
As the conditions in Definition \ref{definition gtclasses} define GT-canonical classes uniquely, it is sufficient to prove that $\eta_p(p)= \Lambda_p^-$ and that for all $q \in M^\T \setminus \{p\}$ with $\lambda_q \leq \lambda_p$ we have $\eta_p(q)=0$. The first assertion follows from observing that the normal bundle of $F_p$ at $p$ coincides with the negative normal bundle of $\mu$ at $p$.
 To prove the second one we use Lemma \ref{lemma Fp bigger index}, which implies that for any $q \in M^\T \setminus \{p\}$ with $\lambda_q \leq \lambda_p$ we must have $q \in M^\T\setminus F_p$,
 and at these points $\eta_p(q)=0$ by definition. 
\end{proof}

\begin{rmk} There is an alternative way of proving that $\eta_p=\zeta_p$ in the index increasing case. If the class $\eta_p - \zeta_p$ were nonzero, it would be nonzero at some fixed point $q$. As both $\eta_p$ and $\zeta_p$ vanish on points in $M^\T \setminus \{p\}$ of index smaller or equal to the index of $p$ (see Lemma \ref{lemma Fp bigger index}), the class $\eta_p - \zeta_p$ can only be nonzero at fixed points $q$ with $\lambda_q>\lambda_p$. Since GT-canonical classes form a basis of $H_\T ^*(M,\Z)$, from the previous observation it follows that $\eta_p-\zeta_p=\sum_{q; \lambda_q>\lambda_p} c_q \zeta_q$, where $c_q \in H_\T^*(\p;\Z)$, with $c_s\neq 0$ for some $s$. By comparing the degrees of the classes on the right and on the left hand side of the above equation we obtain a contradiction. 
\end{rmk}

\begin{lemma}\label{lemma below q in hq}
Let $q \in M^\T$ and $(H_q,\omega_{|H_q})$ be the flow down at $q$ (see the definition in Section \ref{section local index}). Let $p$ be a fixed point connected to $q$ by an oriented edge $\overrightarrow{p\,q}$ and such that $\lambda_p+1=\lambda_q$. Denote by $p_1,\ldots,p_{\lambda_p}$ the fixed points connected to $p$ by oriented edges $\overrightarrow{p_j\,p}$. Then $p_1,\ldots,p_{\lambda_p} \in H_q$.
\end{lemma}

\begin{proof}
Note that the GKM graph for $H_q$ is a graph of valency $\lambda_q$. Therefore exactly $\lambda_q$ of the points connected to $p$ must be in $H_q$. 
Take any  $r \neq q$ connected to $p$ by an edge $\overrightarrow{p\,r}$. There are $n-1-\lambda_p=n-\lambda_q$ such points.
Then $\lambda_r >\lambda_p=\lambda_q-1$, thus $\lambda_r \geq \lambda_q$. Lemma \ref{lemma Fp bigger index} applied to $-\mu$ and $H_q$ (instead of $\mu$ and $F_q$) gives that $r \notin H_q$. Therefore these $n-\lambda_q$ points connected to $p$ are not in $H_q$. It follows that the remaining $\lambda_q$ points, $q$ and the $\lambda_q-1=\lambda_p$ points $p_1,\ldots,p_{\lambda_p}$ directly below $p$, must be in $H_q$.
\end{proof}
To continue we need to recall some definitions from \cite{GT}. Given a weight $w \in \ell^*$ in the weight lattice of $\mathfrak{t}^*$, and a generic $\xi\in \mathfrak{t}$, the projection which sends
$X \in \mathfrak{t}^*$ to $X - \frac{ X (\xi) }{ w (\xi) }w \in \xi^{\perp} \subset \mathfrak{t}^*$ can be extended to be an endomorphism $\rho_{w}$ of $\SS(\mathfrak{t}^*)$, the symmetric algebra of $\mathfrak{t}^*$. 
 Given an edge $\overrightarrow{r_1\, r_2}\in E$ of the GKM graph of $M$, the independence of the set of weights at each fixed point implies that $\rho_{w(\overrightarrow{r_1r_2})}(\Lambda_{r_1}^-) \neq 0$ and $\rho_{w(\overrightarrow{r_1,r_2})}\Big(\frac{\Lambda_{r_2}^-}{w(\overrightarrow{r_1r_2})}\Big) \neq 0$. Therefore the following nonzero elements of the field of fractions of $\SS(\mathfrak{t}^*)$ are well defined for all $\overrightarrow{r_1 r_2}\in E$
$$\Theta(r_1,r_2)= \frac{\rho_{w(\overrightarrow{r_1r_2})}(\Lambda_{r_1}^-)}{\rho_{w(\overrightarrow{r_1r_2})}\Big(\displaystyle\frac{\Lambda_{r_2}^-}{w(\overrightarrow{r_1r_2})}\Big)}.$$
In \cite[Theorem 1.6]{GT} Goldin and Tolman prove that $\Theta(r,r') \in \Z \setminus \{0\}$ for all edges $\overrightarrow{r_1 r_2}\in E$ with $\lambda_{r_2}=\lambda_{r_1}+1.$
(Note that the assumption that the difference of the indices is $1$, though not explicitly stated in their theorem, is implied and required.)
Moreover, they prove the following formula for computing the restriction of a GT-canonical class to a fixed point. 
Let $(V,E_{can})$ be the subgraph of the GKM graph $\Gamma=(V,E)$ where $E_{can}=\{e=\overrightarrow{r_1\,r_2}\in E\mid \lambda_{r_2}=\lambda_{r_1}+1\}$.  
Since the oriented GKM graph is index increasing, this implies that $\mu(r_1)<\mu(r_2)$ for every $\overrightarrow{r_1\,r_2} \in E_{can}$, i.e.\ $E_{can}$ is a subset of $E^o$. 
Then for every $p,q\in M^\T$ we have that
\begin{equation}\label{gt theta formula}
\zeta_p(q)=\Lambda_q^- \sum_{\gamma \in \sum_p^q} \prod_{i=1}^{|\gamma|} \frac{\psi(r_i)-\psi(r_{i-1})}{\psi(q)-\psi(r_{i-1})} \frac{\Theta(r_{i-1},r_i)}{w(\overrightarrow{r_{i-1}\,r_i})},
\end{equation}
where $\sum_p^q$ is the set of paths from $p$ to $q$ in $(V,E_{can})$ and $\gamma\in \sum_p^q$ is given by the sequence of vertices $\gamma=(r_0,\ldots,r_{|\gamma|})$; here $|\gamma|$ denotes the length of $\gamma$, i.e.\ the number of edges composing it.
\begin{lemma}\label{lemma index one}
For all $e=\overrightarrow{r_1r_2}\in E_{can}$ we have that $\Theta(r_1,r_2) =1$.
\end{lemma}
Note that the above lemma may not hold for GKM manifolds which are not toric (see \cite[Example 5.2]{GT})
\begin{proof}
Recall that $W_r^+$ denotes the set of weights of the edges ending at $r$. 
Observe that $\rho_{w}$ sends $w$ to $0$, so it is enough to prove that there exists a bijection $\theta \colon W_{r_1}^+ \rightarrow W_{r_2}^+\setminus \{w(\overrightarrow{r_1r_2})\}$ such that $\theta(w)-w= m \cdot w(\overrightarrow{r_1r_2})$ for some $m \in \Z$ that depends on $w$. This is equivalent to proving that the weights of $\T$ representation on $T_{r_2}H_{r_2}$ and $T_{r_1}H_{r_2}$ agree modulo $w(\overrightarrow{r_1r_2})$. The last fact follows from observing that the subtorus $\T' = \textrm{exp}(\{\eta \in \mathfrak{t}\,|\,w(\overrightarrow{r_1r_2})(\eta)=0\})$ fixes the sphere $S^2_{\overrightarrow{r_1 r_2}}=\psi^{-1}(\overrightarrow{r_1 r_2})$, an embedded $\T$ invariant submanifold of $H_{r_2}$ (which,
in turn,  is a smooth $\T$ invariant submanifold of $M$).
Thus the representations of the torus $\T'$ on the normal bundle of $S^2_{\overrightarrow{r_1 r_2}}$ in  $H_{r_2}$ need to agree at $r_1$ and $r_2$.
\end{proof}
The next lemma proves that for each $q \in V_p^+$ there exists a path $\gamma$ whose edges belong to $E_{can}$. 
\begin{lemma}\label{sigmapq=vpp}
Let $p\in M^\T$ and let $\Sigma_p^q$ be defined as before. Then
$q\in V_p^+$ if and only if $\sum_p^q\neq \emptyset$.
\end{lemma}
\begin{proof}
If $\sum_p^q\neq \emptyset$ then clearly $q\in V_p^+$. Vice versa, suppose that $q\in V_p^+$. Assume first that there exists a path from $p$ to $q$ composed by one edge $\overrightarrow{p\,q}\in E^o$.
Note that the GT-canonical class $\zeta_p$ does not vanish when restricted to $q$. Indeed, $\zeta_p(q)=0$ and condition \eqref{GKM cd eq} would imply that $ w(\overrightarrow{p\,q})$ divides $\zeta_p(p)=\Lambda_p^-$, which contradicts the assumption about linear independence of weights at $p$. 
By \eqref{gt theta formula} we conclude that
$\sum_p^q\neq \emptyset$. If the path from $p$ to $q$ is composed by edges $\overrightarrow{p\,r_1}=\overrightarrow{r_0\,r_1},\ldots ,\overrightarrow{r_{m-1}\,r_m}=\overrightarrow{r_{m-1}\,q}$, each of them in $E^o$, then the preceding argument implies that the sets of paths $\sum_{r_0}^{r_1},\ldots,\sum_{r_{m-1}}^{r_m}$
are all nonempty, and so $\sum_p^q$ is nonempty as well.
\end{proof}
\begin{prop}\label{proposition fp is vpplus}
For any $p\in M^\T$, let $\zeta_p$ be the GT-canonical class. Then for any $q \in M^\T$ have
$$\zeta_p(q) \neq 0 \Leftrightarrow q \in V_p^+.$$
Together with Proposition \ref{pd is gt} this implies 
$$F_p^\T=V_p^+.$$
\end{prop}

\begin{proof}
If $\zeta_p(q) \neq 0$ then by \eqref{gt theta formula} the set $\sum_p^q$ is nonempty, hence $q \in V_p^+$.
Now consider $q \in V_p^+$. By Lemma \ref{sigmapq=vpp} we have that $\sum_p^q$ is not empty.  We use formula \eqref{gt theta formula} quoted from \cite{GT} to analyze $\zeta_p(q)$. 
Using Lemma \ref{lemma index one} observe that each summand in \eqref{gt theta formula} for $\zeta_p(q)$ is positive (in the sense that it gives a positive number when evaluated on $\overline{\xi}$), therefore $\zeta_p(q) \neq 0.$
\end{proof}
\begin{corollary}\label{edges in}
 If $q\in F_p^\T$ and $\overrightarrow{q\,q_0}\in E^o$ then $q_0\in F_p^\T$.  
 As a consequence, the normal bundle of $F_p$ at $q$ (denoted by $N(F_p)|_q$) is a subbundle of the negative normal bundle of $\mu$ at $q$ (denoted by $N_q^-$).
\end{corollary}
\begin{proof}
By definition of $V_p^+$ it follows that if $q\in V_p^+$ then $V_q^+\subseteq V_p^+$. 
By Proposition \ref{proposition fp is vpplus} we know that $F_p^\T=V_p^+$, which implies that if $q\in F_p^\T$ then $F_q^\T\subseteq F_p^\T$.
From the definition of $F_q$ it is straightforward to see that if $\overrightarrow{q\,q_0}\in E^o$ then $q_0\in F_q^\T$, and the first claim follows. 
As a consequence we have that $N(F_p)|_q$ splits as a direct sum of line bundles $\mathbb{L}_i$, each of them being the tangent bundle at $q$ of the sphere associated to the edge $\overrightarrow{q_i\,q}\in E^o$, for some $q_i\notin F_p$. This implies the second claim.
 \end{proof}
\subsection{The proof of Theorem \ref{main cohomology}.}

We are now ready to prove Theorem \ref{main cohomology}. It follows immediately from the Proposition below.

\begin{prop} \label{prop classes equal}
Let $(M,\omega,\T,\psi)$ be a symplectic toric manifold, together with 
a choice of a generic component of the moment map $\mu=\psi^{\overline{\xi}} \colon M\to \R$, not necessarily index increasing. Then for each $p\in M^\T$
the i-canonical class $\tau_p$ exists, and is equal to the equivariant Poincar\'e dual $\eta_p$ to flow-up submanifold $F_p$.
\end{prop}
If the moment map $\mu=\psi^{\overline{\xi}} \colon M\to \R$ is index increasing (thus GT-canonical classes exist) then the above proposition, together with Proposition \ref{pd is gt}, implies that for each $p\in M^\T$
the following equivariant cohomology classes are the same:
\begin{itemize}
\item the Poincar\'e dual $\eta_p$ to flow-up submanifold $F_p$;
\item the GT-canonical class $\zeta_p$;
\item the i-canonical class $\tau_p$.
\end{itemize}

\begin{proof}
Fix $p\in M^\T$. We will show that the equivariant Poincar\'e dual $\eta_p$ to the flow-up submanifold $F_p$ satisfies: $\ii_p(\eta_p)=1$ and $\ii_q(\eta_p)=0$ for each $q\in M^\T\setminus \{p\}$. Since $\eta_p$ is also a Kirwan class, this proves that $\eta_p$ is the i-canonical class at $p$. 
As $\eta_p(p)=\Lambda_p^-$, Remark \ref{index prop coh} implies that $\ii_p(\eta_p)=1$.
For $q \in M^\T \setminus F_p$ we have $\eta_p(q)=0$, thus $\ii_q(\eta_p)=0$. 
Consider a point $q \in F_p \setminus \{p\}$. 
Let $q_1,\ldots, q_{\lambda_q},\ldots, q_n \in M^\T$ be the fixed points connected to $q$ by an edge in the GKM graph of $M$, and let $w_1, \ldots,w_{\lambda_q}, \ldots,w_n$ denote the weights on the corresponding oriented edges, with $\Lambda_q^-=\prod_{j=1}^{\lambda_q}w_j.$
By definition of equivariant Poincar\'e dual we have that
$$\eta_p(q)=\prod_{\stackrel{1 \leq j \leq n}{q_j \notin F_p}} w_j=\prod_{\stackrel{1\leq j \leq \lambda_q}{ q_j \notin F_p }} w_ j\, \cdot \, \prod_{\stackrel{\lambda_q+1 \leq j \leq n}{q_j \notin F_p}}w_j\,.$$
To calculate the index $\ii_q(\eta_p)$ we use the algorithm and the notation of Section \ref{section local index}. Observe that
\begin{align*}
f_0&=\prod_{\stackrel{1\leq j \leq \lambda_q}{ q_j \notin F_p }} (w_ j+w_0)\, \cdot \, \prod_{\stackrel{\lambda_q +1\leq j \leq n}{q_j \notin F_p}}w_j,\\
f_i&=
\begin{cases} 0 & \textrm{ if }q_i \notin F_p\\ 
\displaystyle\prod_{\stackrel{1\leq j \leq \lambda_q}{ q_j \notin F_p }} (w_ j-w_i)\, \cdot \, \prod_{\stackrel{\lambda_q+1 \leq j \leq n}{q_j \notin F_p}} w_j & \textrm{ if }q_i \in F_p
\end{cases}
\end{align*}
Therefore the local index $\ii_q(\eta_p)$ is equal to
$$ \ii_q(\eta_p)=\frac{\displaystyle \prod_{\stackrel{1\leq j \leq \lambda_q}{ q_j \notin F_p }} (w_ j+w_0)\, \cdot \, \prod_{\stackrel{\lambda_q+1 \leq j \leq n,}{q_j \notin F_p}}w_j}{\displaystyle\prod_{1\leq j\leq \lambda_q}(w_j+w_0)}
+ \displaystyle\sum_{\stackrel{1 \leq i \leq n}{q_i \in F_p}}\,\frac{\displaystyle\prod_{\stackrel{1\leq j \leq \lambda_q}{ q_j \notin F_p }} (w_ j-w_i)\, \cdot \, \prod_{\stackrel{\lambda_q+1 \leq j \leq n}{q_j \notin F_p}} w_j}{(-w_i-w_0)\displaystyle\prod_{\stackrel{1\leq j \leq \lambda_q}{j \neq i}}(w_j-w_i)}$$
$$=\left( \displaystyle\prod_{\stackrel{\lambda_q+1 \leq j \leq n}{q_j \notin F_p}}w_j\right) \left( \frac{1}{\displaystyle\prod_{\stackrel{1\leq j \leq \lambda_q}{ q_j \in F_p} } (w_ j+w_0)}+
 \sum_{\stackrel{1 \leq i \leq n}{q_i \in F_p}}\,\frac{1}{(-w_i-w_0)\displaystyle\prod_{\stackrel{1\leq j \leq \lambda_q}{j \neq i,\,q_j \in F_p} }(w_j-w_i)} \right)$$
$$=\left(\displaystyle\prod_{\stackrel{\lambda_q+1 \leq j \leq n}{q_j \notin F_p}}w_j\right)  \ii (\mathbf{1}_{\C \pp ^s}),$$
where $s$ is the number of weights $w_1,\ldots,w_{\lambda_q}$ 
appearing in the representation of $\T$ on $T_qF_p$.
Note that $s\neq 0$ because if $s=0$ then $q$ would be the minimum of $\mu$ on $F_p$ (c.f. Lemma \ref{lemma path to minimum}) which contradicts our assumption that $q \neq p$. 
As $\ii (\mathbf{1}_{\C \pp ^s})=0$ for $s>0$ it follows that $\ii_q(\eta_p)=0$.
\end{proof}

\appendix
\section{An explicit description of the Kirwan map}
The map $\kappa$ used in the definition of local index is in fact the surjective Kirwan map relating the equivariant K-theory or cohomology ring of a manifold $X$ with that of the reduced spaces. Indeed, below we describe a combinatorial algorithm for calculating the Kirwan map, and the reader can compare it with the combinatorial
algorithm for calculating the local index in Section \ref{section local index}. 
For simplicity we only deal with the equivariant cohomology setting.  

Suppose that $X=X^{2d}$ is a $2d$-dimensional symplectic manifold equipped with an effective Hamiltonian toric action of torus $\T=T^d$. Let $\psi' \colon X \rightarrow Lie(T^d)^* \cong \R^d$ be a choice of moment map. Choose a subtorus $T^k \hookrightarrow \T$, $k <d$, and consider the induced  action of $T^k$ on $X$. Let $\pi \colon Lie(T^d)^* \rightarrow Lie(T^k)^*$ be the map induced by the inclusion $T^k \hookrightarrow T^d$. Then $\varphi = \pi \circ \psi' \colon X \rightarrow \R^k$ is a moment map for this action. Take any regular value $a$ of the function $\varphi$. Then $X_{red}:=\varphi^{-1}(a) /T^k$ is a symplectic toric orbifold. If the $T^k$ action on $\varphi^{-1}(a)$ is free then $\varphi^{-1}(a) /T^k$ is a manifold. Moreover it is equipped with a Hamiltonian action of the residual torus $\mathbb{\KK}:=T^d/T^k$. 

By a theorem of Kirwan the following map (called Kirwan map)
 $$\kappa \colon H^*_{\T}(X;\Z)\twoheadrightarrow H^*_\T(\varphi^{-1}(a);\Z)  \simeq H^*_{ \KK }(X_{red};\Z)$$
is surjective.
We describe $\kappa$ explicitly in the situation that appears in our algorithm for calculating the local index (where $X=H_q \times S^2$ and $X_{red}=\widetilde{H}^{\epsilon}_q$), namely when:
\begin{itemize}
 \item $k=1$ so $T^k\cong S^1$, henceforth denoted $\SS^1$ to avoid confusion, 
 \item $X_{red}$ is a manifold and 
 \item the level of the reduction $a$ is close to the maximum of $\varphi$, i.e.\ the hyperplane $\pi^{-1}(a)$ cuts the moment map image of $X$ close to the fixed point $q_0$ of $X$ where $\varphi$ attains its maximum. 
 \end{itemize}
 
The moment map image of $X_{red}$ is the intersection of $\psi'(X)$ with the affine hyperplane $ \pi^{-1} (a) $ in $Lie(T^d)^*$. 
The fixed points of $X_{red}$ correspond to the points of intersection of this affine hyperplane with the edges of the $1$-skeleton of the moment polytope of $X$; the set of these edges is denoted by $E$. 
An example is presented in Figure \ref{fig:firstkirwanmapexample}. The weights of the $T^d=T^2$ action at $p_0$ are $w_1$ and $w_2$. The affine hyperplane is perpendicular to the vector $v=w_1 +w_2$. 

Recall the description of the kernel of Kirwan map from the work of Goldin \cite{G} and Tolman-Weitsman \cite{TWquotients}, and observe that in our situation any class in $H^*_{\T}(X;\Z)$ which has value $0$ when restricted to $q_0$ is in the kernel. Therefore $$\kappa(\alpha) =\kappa( \alpha(q_0) \cdot \mathbf{1})$$ for any $\alpha$ in $H^*_{\T}(X;\Z)$. This reduces our problem to analyzing only the classes of the form $f \cdot \mathbf{1}$, with $f \in H_\T^*(\p;\Z)$. 

We describe $\kappa(f \cdot \mathbf{1})$ by calculating its restrictions to the fixed points $X_{red}^{\KK}$. Let $p_i \in X_{red}^{\KK}$ be any fixed point. Denote by $q_i,q_0\in X^{\T}$ the fixed points in $X$ connected by an edge $\overrightarrow{q_i q_0}\in E$ such that $p_i$ is the intersection of the edge $\overrightarrow{q_i q_0}$ with the affine hyperplane $ \pi^{-1} (a) $. Denote by $S_i^2$ the sphere in $X$ corresponding to the edge $\overrightarrow{q_i q_0}$ and by $H_i$ the subtorus of $\T$ fixing $S_i^2$, i.e. $H_i=\exp (\{\xi\,|\, w(\overrightarrow{q_i q_0}) (\xi)=0\})$. Note that $S_i^2 \cap  \varphi^{-1} (a)$ is a circle, denoted by $\mathcal{C}_i$, equipped with a free $\mathbb{S}^1$-action and that 
$$\mathcal{C}_i /\SS^1=\{p_i\}.$$

To find the weights of the $\KK$ action on $T_{p_i}X_{red}$ we proceed as in \cite[Example 3]{GH}. Observe that $H_i$ is complementary to $\SS^1$ in $\T$, so $H_i \cong \T/\SS^1=\KK$ (this follows from the assumption that the reduced space is a manifold). Therefore the $\T/\SS^1$ action on $T_{p_i}X_{red}$ is isomorphic to the $H_i$ action on this space. The weights of the $\KK=\T/\SS^1$ action on $T_{p_i}X_{red}$ are obtained by projecting the $\T$ weights at $q_0$ to $\mathfrak{h}_i^*=Lie(H_i)^*$. (Using $q_i$ instead of $q_0$ gives the same result as the $\T$ weights at $q_i$ differ from those at $q_0$ by a multiple of $w(\overrightarrow{q_i q_0})$, so the difference vanishes after applying the projection to $\mathfrak{h}_i^*$). Note also that the weights of the $\T/\SS^1$ action on $T_{p_i}X_{red}$, together with $w(\overrightarrow{q_i q_0})$, form a $\Z$-basis of the lattice $\ell^*$.

We need to find the image of $f \cdot \mathbf{1} \in H^*_{\T}(X;\Z)$ under the composition
$$H^*_{\T}(X;\Z) \rightarrow H^*_{\T}(\mathcal{C}_i;\Z) \stackrel{\cong}{\rightarrow} H_{\T/\SS^1}^*(\{p_i\};\Z)=H_{H_i}^*(\p;\Z) ,$$
where the first map is induced by the inclusion $\mathcal{C}_i \hookrightarrow X$ and
 sends  $f \cdot \mathbf{1} \in H^*_{\T}(X;\Z)$ to  $f \cdot \mathbf{1} \in H_\T^*(\mathcal{C}_i;\Z)$. 
As for the second map, on classes of the form $f \cdot \mathbf{1}$, with $f \in H^*_{\T}(\p;\Z)$, it acts exactly as the map $H^*_{\T}(\p;\Z) \rightarrow H^*_{H_i}(\p;\Z)$ we used above to find the weights. Note that the value $\kappa( f \cdot \mathbf{1}) (p_i)$ is in the $\Z$-span of the weights of the $\T/\SS^1$ action on  $T_{p_i}X_{red}$, as it should be.

In conclusion, the procedure for finding the value of the restriction of $\kappa(\alpha)$ to $p_i$ is the following:
\begin{itemize}
\item Present the value $\alpha(q_0) \in H^*_{\T}(\p;\Z)$ in the basis consisting of the $\T/\SS^1$ weights at  $T_{p_i}X_{red}$ and of the weight $w(\overrightarrow{q_i q_0})$.
\item Map such value to $H_{ \KK}^*(\p;\Z)\cong H^*_{H_i}(\p;\Z)$ by sending the weight $w(\overrightarrow{q_i q_0})$ to $0$. 
\end{itemize}
The result is $\kappa(\alpha)$ restricted to the point $p_i$.
Note that starting from $\alpha(q_i)$ instead of $\alpha(q_0)$ gives exactly the same result because $\alpha(q_0)$ and $\alpha(q_i)$ differ by a multiple of the weight $w(\overrightarrow{q_i q_0})$ (see \eqref{GKM cd eq}). By repeating the same argument for each fixed point of $X_{red}$ we obtain the image of $\kappa(\alpha)$ in $H_{\KK}^*(X_{red}^{\KK})$.

In the example in Figure \ref{fig:firstkirwanmapexample}, $\pi^{-1}(a)$ 
 is generated by the vector $(1,-1)$. At $p_1$ we get a $\Z$-basis  $\{w_1-w_2=(1,-1), \,w_1=(1,0)\}$, while at $p_2$ we get a $\Z$-basis  $\{w_2-w_1=(-1,1),\, w_2=(0,1)\}$.
We will find the image of the class $\alpha$ presented in black on the right of Figure \ref{fig:firstkirwanmapexample} (with $\alpha_j=\alpha(q_j)$). The torus $\KK$ is $1$ dimensional, so the dual of its Lie algebra can be identified with $\R[x]$, where $x=w_1-w_2$. At the point $p_1$ our procedure applied to $\alpha_0$ gives 
$$w_1+w_2= -(w_1-w_2)+2w_1 \rightarrow  -(w_1-w_2)+0=-x.$$ 
(Observe that if we use  $\alpha_1$ instead of $\alpha_0$ we indeed get the same result: 
$4w_1+w_2= -(w_1-w_2)+5w_1 \rightarrow -(w_1-w_2)+0=-x.$)
 At the point $p_2$ we get 
 $$w_1+w_2= (w_1-w_2)+2w_2 \rightarrow (w_1-w_2)+0=x.$$
 Therefore $\kappa(\alpha)$ (presented in blue) restricted to $p_1$ gives $-x$, and restricted to $p_2$ gives $x$.
\begin{figure}[htbp]
\begin{center}
\psfrag{p0}{\small{$q_0$}}
\psfrag{p1}{\small{$q_1$}}
\psfrag{p2}{\small{$q_3$}}
\psfrag{p3}{\small{$q_2$}}
\psfrag{p1p}{\small{$p_1$}}
\psfrag{p2p}{\small{$p_2$}}
\psfrag{w1}{\small{$w_1$}}
\psfrag{w2}{\small{$w_2$}}
\psfrag{pi}{{$\pi$}}
\psfrag{l}{Lie$(T^k)^*$}
\psfrag{w12}{\small{$w_1-w_2$}}
\psfrag{xr}{\tiny{$x=w_1-w_2$}}
\psfrag{ww0}{\small{$\alpha_0=w_1+w_2$}}
\psfrag{ww1}{\small{$\alpha_1=4w_1+w_2$}}
\psfrag{ww2}{\small{$\alpha_3=4w_1+7w_2$}}
\psfrag{ww3}{\small{$\alpha_2=w_1+7w_2$}}
\psfrag{x}{{\color{blue}$x$}}
\psfrag{mx}{{\color{blue}$-x$}}
\psfrag{a}{$a$}
\includegraphics[width=14cm]{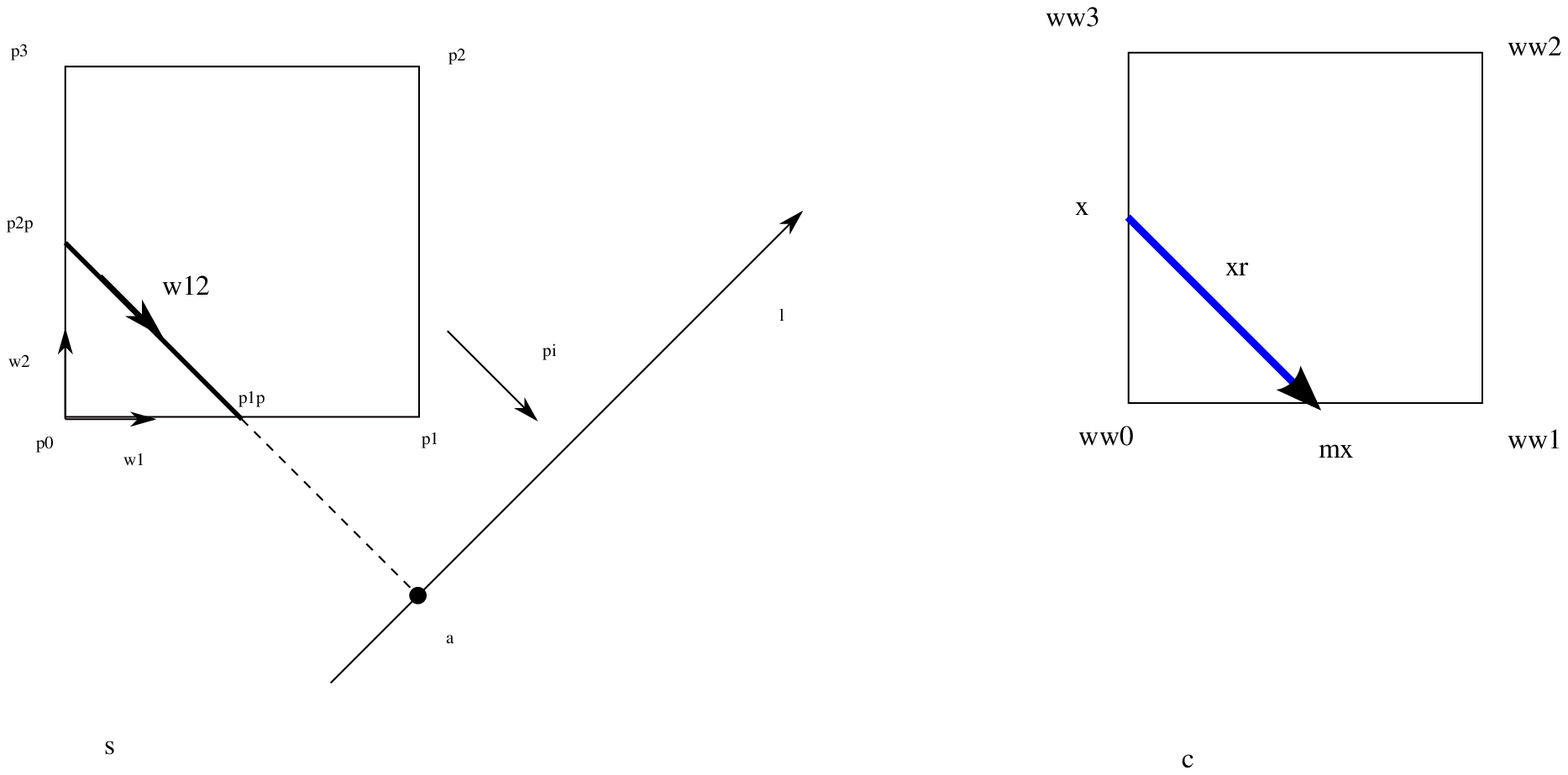}
\caption{An example of a computation of the Kirwan map.}
\label{fig:firstkirwanmapexample}
\end{center}
\end{figure}

Our second example of computation of the Kirwan map goes back to the situation presented in Figure \ref{figure cutting}, i.e.\ the situation we encounter while calculating the local index. 
Given a class $\tau \in H_\T^*(H_q; \Z)$ we want to calculate the class $\kappa(\widetilde{r}(\tau))$ in $\mathcal{K}_{\widetilde{\T}_{q}}((\widetilde{H}^{\epsilon}_q)^{\widetilde{\T}_{q}})$.
\begin{figure}[htbp]
\begin{center}
\includegraphics[width=12cm]{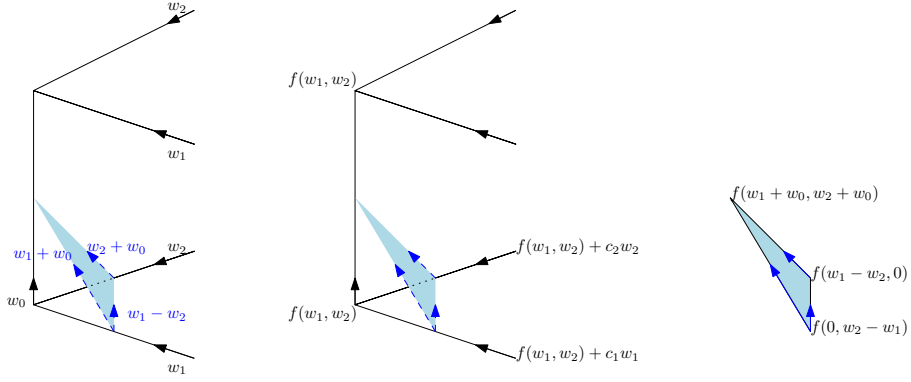}
\caption{A calculation of the Kirwan map in the situation encountered in the definition of local index.}
\label{fig:generalkirwan}
\end{center}
\end{figure}
Figure \ref{fig:generalkirwan} consists of three pictures: the weights around the fixed point $q_0$, the values of $\widetilde{r}(\tau)$ at the fixed points of $H_q \times S^2$ in the neighborhood of $q_0$, and the values of $\kappa(\widetilde{r}(\tau))$ at the fixed points of $\widetilde{H}^{\epsilon}_q$.
The values of $\kappa(\widetilde{r}(\tau))$ at $p_0,\, p_1,\,p_2$, respectively, are calculated in the following way:
\begin{align*}
f(w_1,w_2)&=f(w_1+w_0,w_2+w_0) \rightarrow f(w_1+w_0,w_2+w_0),\\ 
f(w_1,w_2)&=f(w_1, (w_2-w_1)+w_1)  \rightarrow f(0,w_2-w_1),\\
f(w_1,w_2)&=f((w_1-w_2)+w_2, w_2) \rightarrow f(w_1-w_2,0).
\end{align*}


\end{document}